\documentclass[12pt,a4paper]{amsart}
\usepackage{amsfonts}
\usepackage{amsthm}
\usepackage{amsmath}

\usepackage{amscd}
\usepackage[latin2]{inputenc}
\usepackage{t1enc}
\usepackage[mathscr]{eucal}
\usepackage{graphicx}
\usepackage{graphics}
\usepackage{mathrsfs}
\usepackage{pict2e}
\usepackage{epic}
\numberwithin{equation}{section}
\usepackage[margin=2.9cm]{geometry}
\usepackage{epstopdf} 
\usepackage{xcolor}
\usepackage{commath}
\usepackage[rightcaption]{sidecap}
 
\usepackage{tikz-cd}

\usepackage{graphicx, subcaption, float}
\usepackage[rightcaption]{sidecap}

 \usepackage{graphicx}
\usepackage{graphicx}
\graphicspath{ {images/} }

\theoremstyle{plain}

 \theoremstyle{definition}

\newtheorem{?}[Th]{Problem}

\usepackage{hyperref}
\hypersetup{
	colorlinks,
	linkcolor={blue!50!black},
	citecolor={red!50!black},
	urlcolor={blue!80!black}
}

\newtheorem*{theorem*}{Theorem}
\newtheorem{theorem}{Theorem}[section]
\newtheorem{proposition}[theorem]{Proposition}
\newtheorem{property}[theorem]{Property}
\newtheorem{claim}[theorem]{Claim}

\newtheorem{lemma}[theorem]{Lemma}
\newtheorem{observation}[theorem]{Observation}

\newtheorem{question}[theorem]{Question}
\newtheorem*{conjecture*}{Conjecture}
\newtheorem{definition}{Definition}

\theoremstyle{definition}
\newtheorem{remark}[theorem]{Remark}

\newcommand{\wth}{\widetilde{\theta}}
\newcommand{\wds}{\widetilde{ds}}
\newcommand{\wg}{\widetilde{g}}
\newcommand{\C}{\mathcal{C}}

\newcommand{\RR}{\mathbb{R}}
\newcommand{\CC}{\widehat{\mathcal{C}}}
\newcommand{\LE}{\mathcal{L}}
\newcommand{\LL}{\Tilde{\mathcal{L}}}
\newcommand{\G}{\mathcal{G}}
\newcommand{\TG}{\widetilde{G}}
\newcommand{\GG}{\widetilde{\mathcal{G}}}
\newcommand{\mt}{\widetilde{M}}
\newcommand{\F}{\mathcal{F}}
\newcommand{\FF}{\widetilde{\mathcal{F}}}
\newcommand{\HH}{\mathbb{H}}

\newcommand{\A}{\mathcal{A}}

\newcommand{\T}{\mathcal{T}}
\newcommand{\M}{\mathcal{M}}
\newcommand{\N}{\mathcal{N}}

\newcommand{\MM}{\widetilde{\mathcal{M}}}
\newcommand{\NN}{\widetilde{\mathcal{N}}}
\newcommand{\TT}{\mathbb{T}^2}

\newcommand{\OM}{\widehat{M}}
\newcommand{\TM}{\widetilde{M}}

\newcommand{\OA}{\widehat{\mathcal{A}}}
\newcommand{\OT}{\widehat{T}_0}
\newcommand{\OOT}{\widehat{T}}
\newcommand{\OTT}{\widehat{\mathcal{T}}}
\newcommand{\OG}{\widehat{G}}
\newcommand{\OGG}{\widehat{\mathcal{G}}}
\newcommand{\D}{\mathcal{D}}
\newcommand{\q}{\bar{q}}
\newcommand{\og}{\mathfrak{o}}
\newcommand{\h}{\widetilde{h}}

\usepackage{mathtools}
\begin{document}

\title{New Classes of Quasigeodesic Anosov Flows in $3$-Manifolds}

\author{Anindya Chanda \and S\'ERGIO R. FENLEY}
\thanks{S. Fenley's research partially supported by Simons foundation 637554, by National Science
Foundation grant DMS-2054909, and by the Institute for Advanced
Study.}

\address{Department of Mathematics\\ Florida State University \\ 1017 Academic Way\\Tallahassee, FL, 32304, United States} 
\email{ac17t@my.fsu.edu}

\address{Department of Mathematics\\ Florida State University \\ 1017 Academic Way\\Tallahassee, FL, 32304, USA  
and Institute for Advanced Study, Princeton, NJ 08540, USA}

\email{fenley@math.fsu.edu}



\begin{abstract} Quasigeodesic behavior of flow lines  is a very useful property
in the study of Anosov flows. Not every Anosov flow in dimension three is quasigeodesic. In fact up to orbit equivalence, the
only previously known examples of quasigeodesic Anosov flows were 
suspension flows. In this article, we prove that a new class of examples are quasigeodesic. 
These are
the first examples of quasigeodesic Anosov flows on three manifolds
that are neither Seifert, nor solvable, nor hyperbolic.
In general, it is very hard to show that a given flow in quasigeodesic, 
and in this article we provide a new method to prove that an Anosov flow is quasigeodesic. 
\end{abstract}

\maketitle
\section*{Introduction}
 A flow on a manifold  is called \textit{quasigeodesic} if its orbits are uniformly efficient up to a bounded multiplicative
and additive error in measuring distances when lifted to the universal cover.
Quasigeodesics are extremely important for example in hyperbolic 
manifolds. This is because the Morse Lemma says that on a hyperbolic space, any
quasigeodesic is within a bounded distance from a length minimizing geodesic
(when lifted to the universal cover) \cite{Gro87,Thu82}. 
The distance depends on the
quasigeodesic constants. 

From the dynamical systems point of view, there are several 
important reasons to study hyperbolic flows, and in
three manifolds these are Anosov and pseudo-Anosov flows \cite{Ano69,KH95,
Mos92b}.
The question as to whether an Anosov or a pseudo-Anosov flow
in a closed hyperbolic three manifold is quasigeodesic
has been studied a lot \cite{Fen94,Fen95,Fen16,Fen22,FM01,Mos92a}.
In addition any suspension flow is a quasigeodesic flow,
in any manifold \cite{Zeg93}.

The question of quasigeodesic behavior
for Anosov flows on non-hyperbolic three manifolds, which are not orbit
equivalent to suspensions, has not been studied at all. The goal of this
article is to start the study of the quasigeodesic property for Anosov
flows in more general three manifolds.

First we analyze  Anosov flows in Seifert fibered
three manifolds and prove:

\begin{theorem}
Let $\Phi$ be an Anosov flow in a closed, Seifert fibered
$3$-manifold. Then $\Phi$ is a quasigeodesic flow.
\end{theorem}

To prove this theorem we first show that, under an appropriate natural metric, flow lines of the geodesic flow on the unit tangent bundle of a hyperbolic surface are globally length minimizing (in the universal cover). An Anosov flow on a Seifert fibered manifold is orbitally equivalent to a finite lift of the geodesic flow on the unit tangent bundle of a hyperbolic surface. By using the orbit equivalence we show that the flow lines of the Anosov flow on a Seifert fibered manifold are length minimizing up to finite bounds when lifted to the universal cover.

We remark that not every Anosov flow is quasigeodesic: 
in \cite{Fen94}, the second author proved that there exist infinitely many examples of Anosov flows on three dimensional hyperbolic manifolds which are not quasigeodesic. 
In addition for a flow on a general manifold $\M$, there might exist a Riemannian metric such that all the flow lines are geodesic (a differential
geometric condition), but that does not guarantee that the flow lines are quasigeodesic. 

The main result of this article proves the quasigeodesic property
for Anosov flows in new classes of three manifolds.
These manifolds have non trivial JSJ decomposition \cite{Hem76}.
First recall that the DA operation \cite{Wil70} transforms
a hyperbolic periodic orbit into either an attracting or repelling
periodic orbit.
Franks and Williams \cite{FW80} used this operation to
produce the first examples of non transitive Anosov flows
in dimension three as follows: they did a DA operation on
a closed orbit of a suspension, producing a repelling orbit
and an attractor. They removed a solid torus neighborhood of
the periodic orbit to produce a flow in a manifold with
boundary so that the flow is incoming along the boundary.
They carefully glued this with a copy with a time reversed
flow. Under certain homotopy types of gluings, the resulting
flow is Anosov. These examples revolutionized the study of
Anosov flows in dimension three.

In modern terminology, the manifold with boundary, equipped with the
incoming flow, is called a {\em hyperbolic plug}
\cite{BBY17}. In their article Beguin, Bonatti and Yu 
prove that under very general conditions, gluing hyperbolic
plugs produces Anosov flows. In this article we consider
Anosov flows which we call {\em generalized Franks-Williams
flows}. They are obtained as gluings of hyperbolic plugs
as follows: start with a suspension Anosov flow and do
a DA operation on finitely many periodic orbits. The operations
are done so that either they all produce attracting orbits,  or they all produce
repelling orbits. Then remove a solid torus neighborhood
of each DA orbit. Each plug contains either an attractor or a 
repeller. Glue finitely many of these plugs using the techniques
of \cite{BBY17}. The resulting flow is an Anosov flow
\cite{BBY17}. 

The main result of this article is the following:

\begin{theorem} \label{main0}
Let $\Phi$ be a generalized Franks-Williams Anosov flow
in a closed $3$-manifold $M$. Then $\Phi$ is a uniform quasigeodesic
flow.
\end{theorem}

\begin{remark}This result is new in the sense that the manifolds in question
are neither Seifert, nor hyperbolic, nor solvable. The JSJ decomposition
is not trivial. It is easy to prove that the supporting
manifolds of the hyperbolic plugs in question are atoroidal.
In addition, by a result of Leeb \cite{Lee95}, a Haken manifold with at least one atoroidal piece in the JSJ decomposition can be given a Rimennnian metric with non-positive sectional curvature, and hence the universal cover is $CAT(0)$ w.r.t. the induced path metric. 
Therefore the manifolds considered in Theorem \ref{main0} admit
$CAT(0)$ metrics.

\end{remark}

\subsection{Some ideas on the proof of Theorem \ref{main0}.}

First we mention a big difference from the case that the
manifold is hyperbolic.  
As we remarked previously, a lot of study has been done
on the quasigeodesic property for Anosov flows in hyperbolic
$3$-manifolds. In these manifolds a quasigeodesic satisfies
that in the universal cover it is at a bounded Hausdorff distance from
a geodesic. On hyperbolic manifolds geodesics are globally length minimizing curves
in the universal cover. 
The bounded distance property is strongly connected with
the quasigeodesic property, and in certain situations it is
an intermediate step in proving quasigeodesic behavior.

The manifolds in this article are not hyperbolic.
In particular a quasigeodesic in the universal cover may
not be at a bounded Hausdorff distance from a minimal geodesic.
This happens for example in Euclidean space.
In the examples we study, there might exist quasi-flats in the universal
cover $-$ possible examples are lifts of the gluing tori, and they behave like the Euclidean plane.

In this article, we analyze the flow in each
individual block,  and then analyze how the blocks
are assembled together.
In each block the flow is obtained from a blow up of
a suspension Anosov flow, hence it satisfies the 
quasigeodesic property when restricted to the block.
The much more complicated property is to analyze orbits
that cross the tori gluing different blocks. 
This turns out to be very intricate.
We produced our flows
so that we glue a collection of repellers to a collection
of attractors. It follows that an orbit not contained
in a plug intersects one (and only one) of the gluing tori, and goes
from being near a repeller to being near an attractor.
One potential problem is the following:
it could happen that the segments of an orbit on either
side of the gluing torus may track the torus for a long
time and go in opposite directions.
Lifting to the universal cover one produces a big length
along an orbit, but the distance between two points on the orbit may not be large enough compared to the length of orbit segment connecting them.
In fact if one considers arbitrary gluings
on the gluing tori (and not just the ones generating
Anosov flows as in \cite{BBY17}), then this problem can occur and the flow is not quasigeodesic.

We do a very careful analysis to show that when the resulting
flow is Anosov, then the flow lines are quasigeodesics.
One crucial step is related to the potential problem above:
we show that if a flow line intersect certain regions of a gluing
tori, then the forward half orbit `moves away efficiently'
from the lift of the torus when lifted to the universal
cover. In other words, in the manifold, the forward orbit
cannot keep tracking close to the torus for a long time.
This result is Proposition \ref{key} and it is the key component of the main result.     Of course this good behavior is not true for all orbits
intersecting the torus and there is a bad region as well.
In order to prove the quasigeodesic behavior we have to
consider both forward and backwards half orbits from 
points in the  gluing tori, and how they are pierced 
together.


\section{Preliminaries}
A map $f:(X_1,d_1)\rightarrow (X_2,d_2)$ between two metric spaces $(X_1,d_1)$ and $(X_2,d_2)$ is a \textit{quasi-isometric embedding} if there exist two constants $C>1$ and $c>0$ such that, for any points $p_1,p_2\in X_1$,
$$\frac{1}{C}d_1(p_1,p_2)-c\leq d_2(f(p_1),f(p_2))\leq Cd_1(p_1,p_2)+c$$
A \textit{quasigeodesic} in $(X_1,d_1)$ is a quasi-isometric embedding of an interval in $\RR$ (with the standard metric)
in $(X_1,d_1)$; the interval can be any of the 
forms $[a,b], [a,b), (a,b]$ or $(a,b)\subset \RR$ where $a,b\in \RR\cup\{+\infty, -\infty\}$.
If $a$ or $b$ is contained in the interval, then we assume it is not 
either of $+\infty, -\infty$.

Suppose $\N$ is a closed Riemannian manifold with a Riemannian metric $g$ and let the path metric induced by $g$ be denoted
by $d_{g}(\ ,\ )$. A flow $\Phi_t$ on $\N$ with $C^1$-orbits is called \textit{quasigeodesic}, if each flow line $\gamma$ of the lifted flow $\widetilde{\Phi}_t$ in the universal cover $\NN$ is a quasi-isometric embedding of $\RR$ for some constants $C>1$ and $c>0$. 
The metric in $\RR$ is the path distance along the 
flow line. It is immediate that the quasigeodesic property of a flow line $\gamma$ in $\NN$ is equivalent to the following: 
 there exists $C>1$ and $c>0$ such that for any two points $a,b\in\gamma$, 
 $$\text{length}_{\widetilde{g}}(\gamma_{[a,b]})\leq Cd_{\widetilde{g}}(a,b)+c$$
 where $\gamma_{[a,b]}$ is the flow segment connecting $a,b\in\gamma$, $\widetilde{g}$ is the lift of the Riemannian metric $g$ on $\NN$ and $d_{\widetilde{g}}$ is the path metric on $\NN$ induced from $\widetilde{g}$. 

The definition is independent of the metric as the quasigeodesic property is preserved under quasi-isometric embeddings and  as our underlying manifold is compact, any two metrics in $\NN$ which are lifts from
metrics in $\N$  are quasi-isometric to each other. However
the quasi-isometry constants may change. 

If the same quasi-isometry constants $C>1$ and $c>0$ work for all the flow lines then we say that flow is \textit{uniformly} quasigeodesic. It is not true that every quasigeodesic flow is uniform. Notice however that
in closed hyperbolic manifolds,
 Calegari proved in \cite{Cal06} that every quasigeodesic flow 
is uniformly quasigeodesic.  

The focus of this article is to study
the quasigeodesic behavior of Anosov flows. 
\begin{definition}
A $C^1$-flow $\Phi_t:\M\rightarrow\M$ on a Riemannian manifold $\M$ is $Anosov$ if the tangent bundle $T\M$ splits into three $D\Phi_{t}$-invariant sub-bundles $T\M=E^{s}\oplus E^{0} \oplus E^u$ and there exists two constants $B,b >0$ such that 
	\begin{itemize}
		\item $E^0$ is generated by the non-zero vector field defined by the flow $\Phi_t$;
		\item For any $v\in E^s$ and $t>0$, $$||D\Phi_{t}(v)||\leq Be^{-b t}||v||$$
		\item For any $w\in E^u$ and $t>0$, $$||D\Phi_{t}(w)||\geq Be^{b t}||w||$$
	\end{itemize}
\end{definition}

	The definition is independent of the choice of the Riemannian metric $||.||$ as the underlying manifold $\M$ is compact.
	
	For a point $p\in \M$, we will denote
the flow line through $p$ by $\gamma_p$, i,e  $\gamma_p=\{\Phi_t(p)|t\in \RR\}$. The collection of all flow lines of a flow defines an one-dimensional  foliation on $\M$. For an Anosov flow there are several flow invariant foliations associated to the flow and these foliations play a key role in the study of Anosov flows.

	\begin{property}[\cite{Ano69}]
		For an Anosov flow $\Phi_t$ on $\M$, the distributions $E^u$, $E^s$, $E^{0}\oplus E^{u}$ and $E^{0}\oplus E^{s}$ are uniquely integrable. The associated foliations are denoted by $\F^u$, $\F^s$, $\F^{wu}$ and $\F^{ws}$ 
		respectively and they are called the strong unstable, strong stable, weak unstable and weak stable foliation on $\M$. 
	\end{property}
	
	We conclude this section by introducing the notion of \textit{orbit equivalence} between  two flows $\Phi_t^1$ and $\Phi_t^2$.
	
	\begin{definition}\label{d2}
	Two flows $\Phi_t^1:\M\rightarrow \M$ and $\Phi_t^2: \N \rightarrow \N$ are said to be \textit{orbit equivalent} if there exists a homeomorphism $h:\M\rightarrow \N$ such that there exists a continuous map $\tau:\N\times \RR\rightarrow \RR$ such that $h\circ \Phi^1_t\circ h^{-1}(x)=\Phi^2_{\tau(x,t)}(x)$ for all $x\in\N$. 
    \end{definition}
    
    An orbit equivalence maps orbits to orbits with a possible 
 time change. 
	
	\subsection*{Organization of the article:}
	In section \ref{S2}, we prove that geodesic flow on the unit tangent bundle of a hyperbolic surface is quasigeodesic. Moreover, Anosov flows on Seifert Fibered three manifolds are quasigeodesic. 
	
	In Section \ref{S3}, we describe the construction of generalized Franks-Williams flows and in section \ref{S4}, we describe the Riemannian metric we are going to work with in this article. 
	
	Section \ref{S5} contains the proof of Theorem \ref{main0}; Subsection \ref{plug} proves the key proposition for the proof and Subsection \ref{S5.2} completes the proof.


\section{Anosov flows in Seifert manifolds and quasigeodesic
behavior}\label{S2}
Suppose $\Sigma$ is a hyperbolic surface and $T\Sigma$ denotes its tangent bundle, i,e $T\Sigma=\{(p,v)|p\in\Sigma,v\in T_p\Sigma\}$. The universal cover of $\Sigma$ is the hyperbolic plane, we will consider the \text{ P\'oincare upper-half plane} model $\HH$ in this article, i,e.
$$\HH=\{(x,y)\in\RR^2|y>0\}\text{ with the Riemannian metric }ds^2=\frac{dx^2+dy^2}{y^2}$$

On the tangent bundle $T\Sigma$ we can define the \textit{geodesic vector field} w.r.t. the metric $ds$; this is a classical construction, in this article we follow the notations and the detailed description as in \cite[Chapter 3]{doC88}. 

\begin{definition}
The \textit{geodesic field} is defined to be the unique vector field $G$ on $T\Sigma$ whose trajectories are of the form $(\gamma(t),\gamma'(t))$ where $\gamma$ is a geodesic on $\Sigma$ w.r.t. $ds$.   

The flow $\mathfrak{G}_t$ of the geodesic field is called the \textit{geodesic flow} on $T\Sigma$. 
\end{definition}

Suppose $\widetilde{\mathfrak{G}}_t$ is the lift of the geodesic flow on the universal cover $\widetilde{T\Sigma}=T\HH=\HH\times\RR^2$. We show that if $(\gamma(t),\gamma'(t))$ is a flow line of $\widetilde{\mathfrak{G}}_t$ on $T\HH=\HH\times\RR^2$ then it is a quasigeodesic in $T\HH$.

First we choose an appropriate metric on $\widetilde{T\Sigma}$ which projects down to $T\Sigma$. Consider the projection map $\pi:T\HH\rightarrow \HH$. We can define a metric on $T\HH$ using the projection $\pi$ and the metric $ds$ on $\HH$ as described in \cite[Chapter 3, Exercise 2]{doC88}, here is the detailed description:

Suppose $(p,v)\in T\HH$, consider $\alpha_1:t\rightarrow (p_1(t),v_1(t))$ and  $\alpha_2:t\rightarrow (p_2(t),v_2(t))$ where $p_1(0)=p_2(0)=p\in\HH$ and $v_1(0)=v_2(0)=v\in T_p(\HH)$. Let $V_1=\alpha_1'(0)$ and $V_2=\alpha_2'(0)$. Then define the inner product as

$$\langle V_1,V_2\rangle_{(p,v)}=\langle d\pi(V_1),d\pi(V_2)\rangle_p+\langle\frac{Dv_1}{dt}(0),\frac{Dv_2}{dt}(0)\rangle_p$$

\noindent
where $\langle,\rangle_p$ is given by the metric $ds$ on $\HH$ and $\frac{D}{dt}$ denotes the \textit{covariant derivative} as defined in \cite[Proposition 2.2]{doC88}.

Clearly the metric on $T\HH$ as defined above projects down to $T\Sigma$ as the metric $ds$ on $\HH$ projects down to $\Sigma$. Abusing the notation, we denote metric on $T\Sigma$ by $ds$. 

Next we prove that $\mathfrak{G}_t$ on $T\Sigma$ is a \textit{quasigeodesic flow}. Suppose $(\gamma(t),\gamma'(t))$ is a flow line of $\widetilde{\mathfrak{G}}_t$ on $T\HH$, let $(p,v)$ and $(q,w)$ be two points on $(\gamma,\gamma')$. Consider a curve $\zeta:t\rightarrow(\zeta_1(t),\zeta_2(t))\in T\HH$ on $t\in[0,1]$ such that $\zeta(0)=(p,v)$ and $\zeta(1)=(q,w)$ and $\zeta_2(t)\in T_{\zeta_1(t)}\HH$ for all $t\in[0,1]$. Then

\begin{equation}\label{eqn0}
\begin{split}
\text{length}(\zeta)&=\int_0^1 ||\zeta'(t)||dt\\
    &=\int^1_0\sqrt{||\zeta_1'(t)||^2+||\frac{D\zeta_2(t)}{dt}||^2}dt\\
    &\geq \int^1_0 \sqrt{||\zeta_1'(t)||^2}dt=\text{length}(\zeta_1)
\end{split}
\end{equation}
 
Note that $\zeta_1$ is a curve on $\HH$ connection $p,q\in\HH$. But $\gamma$ is a geodesic on $\HH$ and geodesics on $\HH$ are globally length minimizing on $\HH$, which means
$$\text{length}(\zeta_1)\geq\text{length}(\gamma)\text{ between }p,q\in\HH$$

As $\gamma$ is a geodesic, the covariant derivative of $\gamma$ vanishes by definition, i,e $\frac{D\gamma'}{dt}=0$. Using the fact that $\frac{D\gamma'}{dt}=0$ and the Riemannian metric on $T\HH$, it is easy to verify that,

$$\text{length}(\gamma,\gamma')=\text{length}(\gamma)\text{ between }p,q\in\HH$$

Replacing in \ref{eqn0} we conclude that between $(p,v),(q,w)\in T\HH$,
$$\text{length}(\zeta)\geq\text{length}(\zeta_1)\geq\text{length}(\gamma)=\text{length}(\gamma,\gamma')$$

As the choice of $(\gamma,\gamma')$ and $\zeta$ were arbitrary, the above inequality implies that the flow lines of $\widetilde{\mathfrak{G}}_t$ are globally length minimizing in $T\HH$, a stronger property than being a quasigeodesic.
In other words, we proved the following:

\begin{theorem}\label{gf}
The flow lines of the geodesic flow $\widetilde{\mathfrak{G}}_t$ on $T\HH$ are globally length minimizing.

If $\Sigma$ is a hyperbolic surface then the geodesic flow on $T\Sigma$ is a \text{quasigeodesic} flow. 
\end{theorem}

\subsection{Geodesic flows on unit tangent bundle}

We note that the flow lines of the geodesic flow on $T\Sigma$ are of the form $(\gamma,\gamma')$. As $\gamma$ is a geodesic on $\Sigma$ we get $\frac{d}{dt}\langle\gamma'(t),\gamma'(t)\rangle=0$ i,e $||\gamma'(t)||$=constant. This property allows us to restrict the flow $\mathfrak{G}_t$ on $T\Sigma$ to the unit tangent bundle $S\Sigma$ where
$$S\Sigma=\{(p,v)|p\in\Sigma,v\in T_p\Sigma, ||v||=1\}$$
Similarly we can restrict the flow $\widetilde{\mathfrak{G}}_t$ on $S\HH$, the unit tangent bundle on $\HH$. It is immediate by Theorem \ref{gf} that the flow lines of the geodesic flow on $S\HH$ are globally length minimizing. 

It is clear that $S\HH$ is a cover of $S\Sigma$, though it is not the universal cover. As the flow lines of the geodesic flow on $S\HH$ are globally length minimizing and $S\HH$ is complete, lifts of the flowlinws of the geodesic flow in the universal cover $\widetilde{S\HH}=\HH\times\RR$ are also globally length minimizing, a stronger property than being quasigeodesic. 

So far we have considered the metric $ds^2=\frac{dx^2+dy^2}{y^2}$ on $\HH$, and the geodesics and geodesic flow on a surface completely depend on the choice of Riemannian metric. But geodesic flows associated with 
any two negatively curved metric on a surface are orbit equivalent \cite{Ghy84}. More precisely, there is a homeomorphism between the unit tangent bundles of the respective Riemannian metrices which takes orbits to orbit as described in Definition \ref{d2}.  It is easy to check that any homeomorphism between two compact manifolds gives a quasi-isometry when lifted to the universal covers. In particular, as unit tangent bundles of negatively curved closed surfaces are compact, the orbit equivalence maps are quasi-isometries between the universal covers; and quasi-isometries preserve quasigeodesics. This implies geodesic flow w.r.t. any negatively curved metric on a closed surface is quasigeodesic. We conclude the following theorem,

\begin{theorem}\label{UT}
If $\Sigma$ is a hyperbolic surface, the geodesic flows in the unit tangent bundle $S\Sigma$ is quasigeodesic. More precisely, the flow lines in the universal cover are globally length minimizing w.r.t. the metric $ds$.
\end{theorem}

\subsection{Anosov flows in Seifert manifolds}

Now we prove that any Anosov flow on a Seifert fibered three manifold is quasigeodesic. The following theorem relates Anosov flows on Seifert fibered three manifolds with geodesic flows.

\begin{theorem}\cite{Bar96}\label{SF}
Any Anosov flow on a closed Seifert fibered space 
is orbit equivalent to a
finite lift of a  geodesic flow on a hyperbolic surface.
\end{theorem}

We combine Theorem \ref{UT} and Theorem \ref{SF} to get the following:
  
\begin{theorem}
If $\Phi_t$ is an Anosov flow on a Seifert fibered three manifold $\N$, then $\Phi_t$ is quasigeodesic.
\end{theorem}

\begin{proof}
 By Theorem \ref{SF}, $\Phi_t$ is orbit equivalent to a finite lift of the geodesic flow $\mathfrak{G}_t$ on the unit tangent bundle $S\Sigma$ of a hyperbolic surface $\Sigma$. We denote the finite lift of $S\Sigma$ by $\widehat{S\Sigma}$ and the finite lift of the  geodesic flow by $\widehat{\mathfrak{G}}_t$. 

 Fix a Riemannian metric $g$ on $\N$. Let $\widehat{ds}$ be the
metric on $\widehat{S\Sigma}$, $\widehat{ds}$ which
is the lift of the metric $ds$ as constructed before using the upper-half plane $\HH$. 
This is the metric for which Theorem \ref{SF} holds.
We denote the path metrics induced by the lifts of the metrices, $\widetilde{g}$ and $\widetilde{ds}$ on $\NN$ and $\widetilde{S\Sigma}$ resp., by $d_1$ and $d_2$.

%

Fix an orbit equivalence $h:\N\rightarrow \widehat{S\Sigma}$, let $\widetilde{h}:\NN\rightarrow \widetilde{S\Sigma}$ be a lift of $h$ to the universal covers.
By the compactness of $\N$ and $\widehat{S\Sigma}$, we can fix $\eta_1, \eta_2\text > 0$ such that for any $x,y\in\NN$ lying on the same flow line $\gamma$,

\begin{equation}\label{SF.2}
    \text{if length}_{\widetilde{g}}(\gamma_{[x,y]})\geq \eta_1,\ \text{ then length}_{\widetilde{ds}}(\h(\gamma_{[x,y]}))\geq\eta_2
\end{equation}
Consider any two points $a_1,a_2\in\NN$ such that they are on the same flow line of $\gamma$ of $\widetilde{\Phi}_t$. 
Let
$n \ = \ [\text{length}_{\wg}(\gamma_{[a_1,a_2]})]$.

Let $b_0=a_1, b_2,...,b_{n+1}=a_2$ be the points in $\gamma$
such that $\text{length}_{\widetilde{g}}(\gamma_{[b_i,b_{i+1}]})=\eta_1$ for all $0\leq i\leq n-1$ and $\text{length}_{\widetilde{g}}(\gamma_{[b_{n},b_{n+1}]})< \eta_1$. 
Here $b_n = b_{n+1}$ if $\text{length}_{\wg}(\gamma_{[a_1,a_2]})$
is an integer and $b_n \not = b_{n+1}$ otherwise.
Then we get,

\begin{equation}\label{SF.3}
    \begin{split}
        \text{length}_{\widetilde{g}}(\gamma_{[a_1,a_2]})&=\sum_{0\leq i\leq n}\text{length}_{\widetilde{g}}(\gamma_{[b_i,b_{i+1}]})\\
	&= \left(\sum_{0 \leq i \leq n-1} \eta_1\right) + \text{length}_{\wg}(\gamma_{[b_n,b_{n+1}]})\\
	&\leq \left(\sum_{0 \leq i \leq n-1} \eta_1\right) + \eta_1 \ \ 
	= \left(\frac{\eta_1}{\eta_2} \sum_{0 \leq i \leq n-1} \eta_2\right)+ \eta_1\\
	&\leq \left(\frac{\eta_1}{\eta_2}\sum_{0 \leq i \leq n-1} \text{length}_{\wds}(\h(\gamma_{[b_i,b_{i+1}]})\right) + \eta_1\\  
	&\leq \frac{\eta_1}{\eta_2} \text{length}_{\wds}(\h_(\gamma_{[\h(a_1),\h(a_2)]})) + \eta_1\\
	&= \frac{\eta_1}{\eta_2} d_2(\h(a_1),\h(a_2)) + \eta_1, \ \ \text{ by Theorem }\ref{UT}
    \end{split}
\end{equation}
Finally, as $h:\N\rightarrow \widehat{S\Sigma}$ is a homeomorphism
between compact manifolds, the lifts to the universal
covers induce
quasi-isometries between the universal covers. Hence there exists $\eta_3>1$ and $\eta_4 > 0$ such that the map $\h^{-1}:(\widetilde{S\Sigma},d_{\widetilde{ds}})\rightarrow(\NN,d_{\widetilde{g}})$ is an $(\eta_3,\eta_4)$-quasi-isometry. 

Applying the quasi-isometry $\h$ on $\ref{SF.3}$, we get
$$\text{length}_{\widetilde{g}}(\gamma_{[a_1,a_2])})\ \leq \ \frac{\eta_1}{\eta_2}d_2(\h(a_1),\h(a_2)) + \eta_1\ \leq \ \frac{\eta_1}{\eta_2}(\eta_3 d_1(a_1,a_2)+\eta_4) + \eta_1 $$

Finally, let  $A_0=\frac{\eta_1 \eta_3}{\eta_2}$ and $A_1=\frac{\eta_1 \eta_4}{\eta_2} + \eta_1$. It follows that every flow line of $\widetilde{\Phi}_t$ is a $(A_0,A_1)$-quasigeodesic.
\end{proof}

\section{Construction of Generalized Franks-Williams Flows}\label{S3}

A common way to construct Anosov flows is to assemble building blocks. In general a building block is a compact three manifold with boundary equipped with a non-singular vector field transverse to the
boundary. In their article \cite{BBY17}, the authors have combined many
known `assembling building blocks' techniques under a broad general criteria. The building blocks of these type of examples are  called \textit{hyperbolic plugs} (defined below). The first example of a non-transitive Anosov flow, i,e the Franks-Williams flow \cite{FW80} is a classical example of this type of construction.

A $plug$ is a pair $(\M,V)$ where $\M$ is a compact three manifold with boundary and $V$ is a non-singular $C^1$-vector field on $\M$ transverse to the boundary of $\M$. The vector field induces a flow, denoted by $\mathcal{V}_t$, on $\M$. If $\M$ has non-empty boundary, the flow is not complete, i,e, every orbit is defined on a closed time interval of $\RR$, but not every flow line is defined on the whole $\RR$. We consider the \textit{maximal invariant set} $\Lambda$ of $\mathcal{V}_t$, defined as $\Lambda:=\bigcap_{t\in\RR}\mathcal{V}_t(\M)$. In other words, $\Lambda$ is the collection of all orbits which are defined for the whole $\RR$, equivalently
these orbits do not intersect $\partial\M$. If $\Lambda$ is a hyperbolic set, we say $(\M,V)$ is a $hyperbolic$ plug. Here is the precise definition:

\begin{definition}
A \textit{hyperbolic plug} $(\M,V)$ is a plug whose maximal invariant set $\Lambda$ is \textit{hyperbolic}, which means, for every $x\in\Lambda$, $T_x\M$ splits into three one-dimensional sub-bundles 
$$T_x\M=E^s(x)\oplus\RR V(x)\oplus E^u(x)$$
The bundle $E^s(x)$ (resp. $E^u(x)$) is called the stable (resp. unstable) bundle and there exists a Riemannian metric such that the differential of the time-one map of the flow uniformly contracts (resp. uniformly expands) the vectors of the stable bundle (resp. unstable bundle). The splitting varies continuously on $x\in\Lambda$ and is invariant under the derivative of the flow $\mathcal{V}_t$. 
\end{definition}

The study of hyperbolic dynamics is an enormous area of study, here we recall few of the preliminary properties (as in \cite{BBY17}) required for this article, we refer to \cite{KH95} for details:
\begin{itemize}
    \item for every $x\in\M$, the \textit{strong stable manifold} $W^{ss}(x)$ is defined as follows :
    $$W^{ss}(x)=\{y\in\M|\ d(\mathcal{V}_t(x),\mathcal{V}_t(y))\rightarrow 0\text{ as }t\rightarrow +\infty)\}$$
    \textit{The strong unstable manifold} is defined as the strong stable of the reversed flow $-\mathcal{V}_t$. 
    
    \item The \textit{weak stable} manifold $W^{s}$( resp. \textit{weak unstable} manifold $W^u$) of a point $x\in\M$ is defined as the union of the strong stable manifolds (resp. strong unstable manifolds) of all points on the orbit of $x$.
    
    \item There exists two 2-laminations, denoted by $W^s(\Lambda)$ and $W^u(\Lambda)$, whose leaves are the weak stable and weak unstable manifolds, respectively, of the points of $\Lambda$. The leaves of the laminations are $C^1$-immersed manifolds tangent to continuous plane fields $E^s\oplus\RR V(x)$ and $E^u\oplus \RR V(x)$. 
\end{itemize}

The boundary of $\M$ is partitioned into two disjoint subsets, namely the \textit{exit boundary} $\partial^{out}$ and the \textit{entrance boundary} $  \partial^{in}$, where $\mathcal{V}_t$ points outwards on $\partial^{out}$ and inwards on $\partial^{in}$. If $\partial^{out}=\emptyset$ then $(\M,V)$ is an $attracting$ plug and similarly $\partial^{in}=\emptyset$ implies a $repelling$ plug.

\begin{itemize}

\item If both $\partial^{in}\neq \emptyset$ and $\partial^{out}\neq \emptyset$ then $\Lambda$ is a `saddle'. In that case, the weak stable lamination $W^{s}(\Lambda)$ intersects $\partial^{in}$ transversally and forms a one dimensional lamination $\LE^{s}_{V}=\partial^{in}\cap W^{s}$ on $\partial^{in}$. Similarly, the weak unstable lamination $W^{wu}(\Lambda)$ intersects $\partial^{out}$ in a one dimensional lamination $\LE^{u}_{V}=\partial^{out}\cap W^{u}$. 

\item For an attracting plug, $\Lambda$ is an $attractor$. In this case the weak-stable lamination $W^{s}(\Lambda)$ intersects $\partial^{in}$ in an one dimensional lamination $\LE^{s}_{V}=\partial^{in}\cap W^{s}$.

\item For a repelling plug, $\Lambda$ is a $repeller$. In this case the weak-unstable lamination $W^{u}(\Lambda)$ intersects $\partial^{out}$ in an one dimensional lamination $\LE^{u}_{V}=\partial^{out}\cap W^{u}$.
\end{itemize}

\begin{proposition}[\cite{BBY17}]{\label{BBY}}
Consider a finite collection of hyperbolic plugs, denoted by
$(\M_1, V_1), (\M_2, V_2),..., (\M_n,V_n)$. 
Assume that each of these plugs is either an attracting or
a repelling plug. 
Let $\mathcal{D}^{out}=\sqcup_{1}^{n}\partial^{out}(\M_i)$ and $\mathcal{D}^{in}=\sqcup_{1}^{n}\partial^{in}(\M_i)$. 
Suppose that the laminations $\LE^s_{\M_i}, \LE^u_{\M_j}$ (if they
are non empty)
are filling laminations in the respective boundary components. 
Suppose there exists a diffeomorphism $\Omega:\mathcal{D}^{out}\rightarrow \mathcal{D}^{in}$ such that $\Phi_*(\LE^{s}_{\M_i})$ is transversal to $\LE^{u}_{\M_j}$ on each appropriate component. Then the quotient vector field $\frac{V_1\sqcup V_2\sqcup...\sqcup V_n}{\Omega}$ is Anosov on the quotient manifold $\frac{\M_1\sqcup\M_2\sqcup...\sqcup \M_n}{\Omega}$.
\end{proposition}

Since the plugs are either attractors or repellers, the 
laminations in the boundary are actually foliations. The result 
above is then Proposition 1.1 of \cite{BBY17}.

 In this article we consider a special type of $attracting$ and $repelling$ plugs, which we call \textit{Fanks-Williams type} hyperbolic plug. This type of of construction was first introduced by  Franks and Williams in \cite{FW80}. We construct our plugs using DA bifurcations of hyperbolic automorphisms on two-torus $\TT$. Details of the construction are described below:

\subsection{Construction of the Franks-Williams Type Hyperbolic Plugs:}{\label{Const}} 
Consider a hyperbolic linear automorphism $A$ on the two-torus $\TT$, which is induced by a linear map $\widetilde{A}:\RR^2\rightarrow \RR$ such that $\widetilde{A}$ has two eigenvalues $\lambda>1$ and $\frac{1}{\lambda}<1$. 
On $\TT$ we have a pair of one dimensional foliations, namely the stable $\LE^s$ and unstable $\LE^u$ foliations of the hyperbolic map $A$ as described below:
\begin{itemize}
    \item \textbf{ Unstable foliation $\LE^u:$} $\RR^2$ has a foliation $\LL^u$ by the lines parallel to the eigenvalue direction $\lambda$ and this foliation is $\widetilde{A}$-invariant. Hence $\LL^u$ on $\RR^2$ projects down to a foliation on $\TT$ and it is the the \textit{unstable foliation} $\LE^u$ of $A$ on $\TT$.
    
    \item \textbf{Stable foliation $\LE^s$:} Similarly, the foliation on $\RR^2$ induced by the lines on the $\frac{1}{\lambda}$-direction projects down to the \textit{stable foliation} $\LE^s$ of $A$ on $\TT$.
\end{itemize}

 These two foliations are everywhere transversal to each other on $\TT$. Hence they define a two-frame $\{\mathscr{X},\mathscr{Y}\}$ on the tangent bundle $T\TT$ where $\mathscr{X}(p)$ is  a vector in $T_p\TT$ tangent to the stable direction and similarly, $\mathscr{Y}(p)$ is a vector tangent to the unstable direction in $T_p\TT$. In fact, we can define a new coordinate system  $\{x,y\}$ on $\RR^2$.
 
\vskip .05in
\noindent
 \textbf{A new coordinate system $\{x,y\}$:} fix a basis $\{v_{1/\lambda},v_{\lambda}\}$ on $\RR^2$ where the basis vectors are eigenvectors of the two distinct eigenvalues $\lambda$ and $1/\lambda$. Then the new coordinate system on $\RR^2$ w.r.t. $\{v_{1/\lambda},v_{\lambda}\}$ is denoted by $\{x,y\}$. In this coordinate, $\widetilde{A}$ can be written as $\widetilde{A}(x,y)=(\frac{1}{\lambda} x,\lambda y)$.  We use this coordinate system extensively in the rest of the article.
 
 \vskip 0.2in
The fixed point $(0,0)$ of $\widetilde{A}$ on $\RR^2$ projects to a fixed point of $A$, denoted by $\mathfrak{o}$, on $\TT$. We can change it to a point source or a point sink using the `Derived from Anosov(DA)' bifurcation on a neighbourhood of $\mathfrak{o}$. 
Here we give a quick description of the technique, a detailed description can be found in \cite[section 17.2]{KH95} or in \cite{Wil70}.

Consider two closed disks  $D_1$ and $D_2$ on $\TT$ centered at $\og$ such that $\og\in D_1 \subset \mathring{D}_2$. On $D_2$ we consider the local coordinate system $\{x,y\}$ around $\og$ projected from the coordinates $\{x,y\}$ on $\RR^2$ around $(0,0)$.  W.r.t that coordinates on $D_2$ the fixed point $\og\in\TT$ is represented by $(0,0)$. Then we `blow-up' the fixed point $\og$ using a smooth map $\phi$ as described below: 

\begin{equation*}
    \begin{split}
    \phi(x,y)&=(\theta(x,y)x,  y) \text{ on }D_2\\
    \phi &= Id\text{ on }\TT\setminus D_2
\end{split}
\end{equation*}

In the above description, $\theta(x,y):\TT\to [1,\infty)$ is a smooth map such that, on $\TT\setminus D_2$ we have $\theta(x,y)=1$ and near the point $\og$, the map
$\theta(x,y)$ is large enough to counteract the contraction along the $x$-lines. Then $\Phi=A \circ \phi$ is a map with a point source at $\og$. 
Notice that the coordinates $(x,y)$ make sense in a neighborhood
of $\og$, but clearly one cannot have global coordinates in $\TT$.
Still the equations above make sense.

\begin{property}[\cite{Wil70,Sma67}]
The new map satisfies the following properties:
\begin{enumerate}
    \item $\Phi=A \circ \phi$ is homotopic to $A$.
    \item The maximal invariant set of $\Phi$ consists of a point source and a one-dimensional hyperbolic attractor, denoted by $\Lambda$.
    \item $A \circ \phi$ preserves the stable foliation $\LE^{s}$ of $A$. More precisely, the attractor of $\Phi$ on $\TT$  is an attracting hyperbolic set, denoted by $\Lambda$.
This induces  a stable foliation in $\TT - \og$,
denoted by $\LE^s(\Lambda)$. The construction is
done so that the leaves of $\LE^s(\Lambda)$ 
are contained in leaves of $\LE^s$.
Only the stable leaf of $\og$ is split into two stable
leaves of $\LE^s(\Lambda)$. All the other leaves are the same.
\end{enumerate}
\end{property}

\begin{remark}
The usual form of blow up is to first apply the hyperbolic
map $A$ and then the blow up $\phi$. 
It is equivalent to what we do here: the inverse $\Phi^{-1}
= \phi^{-1} \circ A^{-1}$ and the contraction under $\phi^{-1}$ in
a neighborhood of 
$\og$ is stronger than the expansion of $A^{-1}$. Hence
$\og$ is an attractor for $\Phi^{-1}$ and there is a 
one dimensional repeller $\Lambda$ for $\Phi^{-1}$.
$\Lambda$ is the attractor for $\Phi$.
We do it in this form, since it is easier to prove later on some
invariance properties of a metric we will be interested in.
\end{remark}

\begin{remark}
We have described above the DA bifurcation to get a point source. Similarly, we can change the fixed point $\mathfrak{o}$ to a sink. In that case the maximal invariant set will consist of a point sink and an one-dimensional repeller and the map $A\circ\phi$ would preserve the foliation $\LE^u$.
\end{remark}
Next consider the suspension manifold 
$$M=\frac{\TT\times \RR}{(q,t)\sim (\Phi(q), t-1)}\text{ for all }q\in\TT\text{ and }t\in\RR$$

The universal cover of $M$, denoted by $\widetilde{M}$, is $\RR^2\times\RR$ equipped with the coordinate system $\{x,y,t\}$ where the $x$-axis and $y$-axis are as described above and $t$-axis is in the vertical direction. The vertical lines induce a natural flow $\psi_t$ on $M$ so that
its lift $\widetilde \psi_t$ to the 
universal cover $\widetilde{M}$ is defined by $\psi_t([q,s])=[q,t+s], q \in \RR^2$. Note that we have a periodic orbit $\C$ of $\psi_t$ homeomorphic to the circle inside $M$ through the fixed point $\og\in \TT$. 

To construct a hyperbolic plug we cut out an open solid torus neighbourhood $N(\C)$ of the periodic orbit $\C$, the new manifold $M_1=M\setminus N(\C)$ is a manifold with boundary and we denote the boundary by $T_1=\partial M_1$, the boundary is homeomorphic to 2-torus. We choose $N(\C)$ in such a way that boundary of $N(\C)$ is a smooth torus embeded in $M$ and the flow lines of $\psi_t$ transversally intersect the boundary of $N(\C)$. 
We will have a further condition on $T_1$ described later.
Finally, we can restrict the flow $\psi_t$ on $M_1$, and the restricted semiflow on $M_1$ will be denoted by $\psi^1_t$.


It is clear from the construction that $(M_1,\psi^1_t)$ is an attracting hyperbolic plug as the flow $\phi_t$ is the suspension flow of a `DA' map with a attractor in the maximal invariant set. To ensure that when another plug is attached to $M_1$ along $\partial M_1$, the semiflows are matched smoothly along the boundary, we attach a collar neighbourhood homeomorphic to $T_1 \times [0,1]$ along $\partial M_1=T_1$ such that $\partial M_1$ is glued with $T_1\times \{0\}$. we call the new manifold $\M_1$, and the boundary component of $\M_1$ is denoted by $\T_1=\partial \M_1$. Now propagate $\psi^1_t$ in $T_1\times [0,1]$ via an isotopy such that the extension of the flow on $T_1\times[0,1]$ is a product flow topologically. We denote the extended flow on $\M_1$ by $\Psi^1_t$.
In the next proposition we sum up the description of the above construction: 

\begin{proposition}{\label{plug}}
The pair $(\M_1, \Psi^1_t)$ as constructed above is a hyperbolic plug where $\partial\M_1=\partial^{in}\M_1$ 
 and its maximal invariant set in $\M_1$ is a hyperbolic attractor. In fact,
we can make the blow up operation on finitely many periodic
orbits to obtain
an attracting hyperbolic plug with a finite number of boundary components. The stable foliation of the hyperbolic attractor inside $\M_1$ intersects each component of $\partial \M_1$ in an one-dimensional foliation with two Reeb annuli.

If we consider a repelling `DA' map in the previous construction, we would get a repelling hyperbolic plug where the unstable foliation of the repelling set inside $\M_1$ intersects each component of $\partial \M_1$ in an one foliation with Reeb components.
\end{proposition}

\begin{definition}
Any attracting or repelling plug with finite number of boundary components as constructed in proposition \ref{plug} is called a \textit{Franks-Williams type} hyperbolic plug in this article.  
\end{definition}

\begin{remark}
The explicit description of the blow up is done for the
orbit which is the suspension of the point $(0,0)$. 
Later computations and results will be done relative to this
orbit. Given a different periodic orbit, one can chose a different
coordinate system so that this different orbit is the one through
the suspension of $(0,0)$. Hence the arguments in this
article work for any collection of blow ups as described above.
 \end{remark}

\subsection{Construction of the example manifolds and the flows on them}{\label{manifold and flow}}

We consider a finite collection of Franks-Williams type plugs, say 
$$\{(\M_1,\Psi^1_t);(\M_2,\Psi^2_t);...;(\M_n,\Psi_t^n)\}$$

\noindent
along with a diffeomorphism $\Omega$ from the collection of exit boundaries $\mathcal{D}^{out}=\sqcup_1^n\partial^{out}_i$ to the collection of entrance boundaries $\mathcal{D}^{in}=\sqcup_1^n\partial^{in}_i$. In this notation, for any plug $(\M_i,\Psi^i_t)$, either $\partial^{out}_i$ or $\partial^{in}_i$ is empty, and the other one is non-empty. If $\partial^{in}_i$ is nonempty, then it may have more than one component, each homeomorphic to a two-torus, and the weak-stable foliation $\LE^{ws}(\Lambda_i)$ of the semiflow $\Psi_t^i$ intersects each component of $\partial^{in}_i$ in a union of two
 Reeb annuli. Similarly if $\partial^{out}_i\neq \emptyset$, each component of $\partial^{out}_i$ intersects the weak-unstable foliation $\LE^{wu}(\Lambda_i)$ in a one dimensional foliation with two Reeb annuli. 

Using the diffeomorphism $\Omega$, we can construct the manifold

$$\N=\frac{\M_1\sqcup \M_2\sqcup...\sqcup \M_n}{\Omega(q)\sim q}$$

\noindent
and the semiflows $\{\Psi^1_t,\Psi^2_t,...,\Psi^n_t\}$ match to produce
a flow $\Psi_t$ on $\M$.
If we consider a diffemorphism $\Omega: \mathcal{D}^{out}\rightarrow \mathcal{D}^{in}$ that transversally maps the one foliations on each component of $\partial^{out}_i$ to the one foliations on the respective components of $\mathcal{D}^{in}$, then by the proposition \ref{BBY}, the flow $\Psi_t$ on $\N$ is \textit{Anosov}.


\begin{definition}
An Anosov flow constructed in the way described above from Franks-Williams type plugs will be called \textit{Generalized Franks-Williams (GFW)} flow in this article.  
\end{definition}

The original construction by Franks and Williams was to construct a hyperbolic plug $(\M_1,\Psi^1_t)$ from an DA map and with one exit boundary component, and attaching it with $(\M_1,-\Psi^1_t)$ (i,e the same manifold equipped with the reversed flow), along the boundaries with a $\pi/2$-rotation.

\section{Riemannian Metric on the Plugs and the whole manifold}\label{S4}

To analyze quasigeodesic behaviour of the flow lines, we will first define a suitable Riemannian metric on our manifold. As $\N$ is a compact manifold, for any two Riemannian metrics $\tilde{g}_1$ and $\tilde{g}_2$ on the universal cover $\NN$ (which project down on $\N$), the identity map on $\NN$ is a quasi-isometry w.r.t. the induced path metrics. As quasigeodesic
behavior 
is a property that is invariant under
quasi-isometries, it is enough to work with one fixed metric.  

We construct a Riemannian metric $\G_i$ on the hyperbolic plugs $\M_i$ for each $i$, and then attach them along the boundary components of the $\M_i$'s using the map $\Omega$ to construct the metric $\G$ on the whole manifold $\M$. We describe the details of the construction of the Riemannian metric
$\G_1$ on $\M_1$ and the same process works for all other plugs $\M_i$'s.

\subsection{Construction of a Riemannian metric $\mathcal{G}_1$ on $\M_1$}{\label{metric}}
Consider the attracting DA map $\Phi:\TT\rightarrow \TT$ and the manifold $M=\TT\times [0,1]/\sim$ as described in the previous section. The universal cover of $M$, $\widetilde{M}=\RR^2\times \RR$
is equipped with the coordinate system $\{x,y,t\}$. The lift of the flow $\psi_t$ to the universal cover $\MM=\RR^2\times\RR$ will
be henceforth denoted by $\widehat{\psi}_t$ (we explain the $\widehat{ }$  notation later). We can also define a three-frame $\{\mathscr{X},\mathscr{Y},\mathscr{T}\}$ on the tangent bundle
$T(\RR^2\times\RR)$ where the vector fields $\mathscr{X}$, $\mathscr{Y}$ and $\mathscr{T}$ are parallel to the $x$-direction, $y$-direction and $t$-direction respectively. In addition, we define the vector field $\mathscr{T}$ as $\mathscr{T}=\frac{d}{dt}\widehat{\psi}_t$.

\begin{remark}
 The $y$-direction on $\RR^2$ were parallel to the strong unstable foliation of $\widetilde{A}$ on $\RR^2$, but the blow-up does not preserve the unstable direction. Hence the $y$-direction does not represent the unstable leaves of the attractor of the $DA$-map $\widetilde{A}\circ \tilde{\phi}$, though the $x$-direction is parallel to the stable leaves of the attractor. 
 \end{remark}
We first construct a metric on $M$ and then restrict it to $M_1=M\setminus N(\C)$. Our convention is that $M_1$ is the complement of the
interior of $N(\C)$ in $M$, so $M_1$ is compact, and with boundary.
We will use an intermediate cover of $M_1$ which will be denoted by
$\OM_1$. 
Consider the repelling periodic orbit $\C$ of $\psi_t$ in $M$ and $N(\C)$ a solid torus neighbourhood of $\C$ as described in subsection 2.1. Let $\widehat{N(\C)}$ 
be the collection of lifts of $N(\C)$ in $\mt = \RR^2\times \RR$.
Then define $\OM_1=(\RR^2\times \RR)\setminus \widehat{N(\C)}$,
again the convention is that we are removing the interior of the sets.
In other words $\OM_1$ is the pullback of $M_1$ under the
cover $\widetilde M \to M$.
This is an infinite cover of $M_1=M\setminus N(\C)$, but it is not the universal cover of $M_1$.
Later in this article, 
we do a lot of the analysis in $M_1$ and $\OM_1$ instead of
$\mt$. For this reason we will denote the metrics using the
hat notation.

If $\OG_1$ is a Riemannian metric on $\mt$ which projects down to $M$, then the deck transformations on $\mt=\RR^2\times\RR$ have to be isometries w.r.t $\OG_1$. The deck transformation group on $\mt$ is generated by the following diffeomorphisms:

\begin{enumerate}
    \item $\Gamma: \RR^2\times\RR \rightarrow \RR^2\times \RR$, $$\Gamma(x,y,t)=(\widetilde{\Phi}(x,y), t-1)$$
    where $\widetilde{\Phi}(x,y)$ is the lift of the map $\Phi=A\circ\phi$ from $\TT$ to $\RR^2$ so that $\widetilde{\Phi}(0,0) = (0,0)$. 
    
    \item Translations by one unit in two horizontal directions w.r.t the Euclidean coordinate system on $\RR^2\times\RR$, i,e
    \begin{enumerate}
        \item $E_1(e_1, e_2,t)=(e_1+1, e_2,t)$
        \item $E_2(e_1, e_2,t)=(e_1, e_2+1,t)$
    \end{enumerate}
\noindent
    where $e_1$ and $e_2$ are given w.r.t the Euclidean co-ordinates on $\RR^2\times\RR$.
    \end{enumerate}
    
    It is enough to construct the metric $\OG_1$ on $\RR^2\times[0,1]$ such that the maps
    $$\Gamma:\RR^2\times\{1\}\rightarrow \RR^2\times \{0\}, \  \ \Gamma(x,y,1)=(\widetilde{\Phi}(x,y),0)=(\tilde{A}\circ \tilde{\phi}(x,y),0)$$ 
    and the translations $E_1$ and $E_2$ restricted on $\RR^2\times [0,1]$ are isometries. 
Notice that the first map is between two dimensional sets, and the
other two are between three dimensional sets.

As before, we consider the coordinate system $\{x,y,t\}$ on $\RR^2\times [0,1]$. The idea to define the metric on $\RR^2\times[0,1]$ is as follows: We
will pick a suitable  metric $g_0$ on the level {t=0} and  consider a family $h_s$ 
of maps, smoothly varying with $s\in[0,1]$ where $h_0=Id$ and $h_1=\widetilde{\Phi}$. Then we will pull-back the metric $g_0$ from the level $\{t=0\}$ to the level $\{t=s\}$ via the map $h_s$.  

First we define a family of maps on $\RR^2$. Consider the neighbourhood $D_2$ with the local co-ordinate system $\{x,y\}$ as defined in the description of the DA map in the previous section. 
Let $\wth$ be the lift to $\RR^2$ of $\theta: \TT \to \RR$.
Now define the family of maps $\tilde{\eta}_s:\RR^2\rightarrow \RR^2$ for $s\in[0,1]$ as follows: 
First define $$B_s(x,y) \ = \ (\lambda^{-s} x, \lambda^s y)$$
\noindent
Notice that $B_1 = \widetilde{A}$.
Also define

\begin{equation*}
    \begin{split}
    \nu_s(x,y) &= \ ((\theta(x,y))^s x,  y) \text{ on }D_2\\
    &= \ (x,y) \text{ on }\TT\setminus D_2.
\end{split}
\end{equation*}
\noindent 
Again notice that $\nu_1 = \phi$.
Let $\tilde{\nu}_s: \RR^2 \to \RR^2$ be the lift of $\nu_s$ so
that $\tilde{\nu}_s(0,0) = (0,0)$. Then define

$$\tilde{\eta}_s(x,y) \ = \ B_s \circ \tilde{\nu}_s (x,y)$$

Now we are ready to define the family of maps $h_s:\RR^2\times\{s\}\rightarrow \RR^2\times\{0\}$:
$$h_s(x,y,s)=(\tilde{\eta}_s(x,y),0)\text{ for }s\in[0,1] $$
\noindent
In general, it is not easy to find the exact formula of the lifts $\tilde{\eta}_s$, except near the point $(0,0)$,
where it is easy to get an explicit formula. 

For each $s$, the map $h_s$ takes the level $\RR^2 \times \{ s \}$
to $\RR^2 \times \{ 0 \}$.
We start with the metric $g_0^2=dx^2+dy^2$ on the level $\RR^2\times \{0\}$ where $dx$( resp. $dy$) measures the length along $x$-directions (resp. $y$-directions). 

Using the family $h_s, 0 \leq s \leq 1$, we can pull back the metric $g_0^2=dx^2+dy^2$ from $\RR^2\times\{0\}$ to each level $\RR^2 \times \{ s \}$.
Hence on $\RR^2\times\{t\}$, we define the pull back metric $g_t=(h_t)^*(g_0)$ where $t\in[0,1]$. 

With the induced differentiable structure in $\RR\times[0,1]$ from inclusion
in $\RR^2\times\RR$ it is easy to see that the 
the metrics on $\RR\times\{t\}$ vary smoothly with $t\in[0,1]$ because of the smoothness of family $h_s, 0 \leq s \leq 1$. 
We now define a metric not only on horizontal vectors, but on all
vectors. At the point $q=(q_1,t)\in\RR^2\times[0,1]$ the metric 
is defined as follows:
\begin{equation}
    \begin{split}
&\OG_1(\mathscr{X},\mathscr{Y}) = g_t(\mathscr{X},\mathscr{Y})
\text{ on the level }\RR^2\times\{t\}\text{ for }t\in[0,1]; \\
& \OG_{1}(\mathscr{T}, a\mathscr{X}+b\mathscr{Y})= 0\text{ for all }a,b\in\RR;\\
& \OG_1(\mathscr{T},\mathscr{T})= 1;
   \end{split}
\end{equation}



\begin{observation}\label{iso}
Then metric $\OG_1$ on $\RR^2\times[0,1]$ as defined above in \ref{metric} is invariant under the the maps $\Gamma$, $E_1$ and $E_2$ as follows:
\begin{enumerate}
    \item Clearly the map $\Gamma:\RR^2\times\{1\}\rightarrow \RR^2\times \{0\}$ defined by $\Gamma(x,y,1)=(\tilde{A}\circ \tilde{\phi}(x,y),0)$ is an isometry as $$\OG_1|_{\RR^2\times\{1\}}=(\tilde{A}\circ \tilde{\phi})^*(\OG_1|_{\RR^2\times\{0\}})$$

   \item 
On the level $\RR^2\times \{0\}$ the metric $g_0^2=dx^2+dy^2$ is invariant under the translations. On each level $\RR^2\times \{t\}$, the metric $(h_t)^*(g_0)$ can be written as 
$$(B_t \circ \tilde{\nu_t})^*(g_0) \ = \ 
(\tilde{\nu}_t)^* ((B_t)^* (g_0))$$
\noindent
Notice first that $B_t$ does not leave invariant
the integer lattice. However $(B_t)^*$ leaves invariant
the metric $-$ it is the solv metric in this setting.
It is clear that  $\tilde{\nu_t}$ leaves invariant
the metric
under integer translations, since it came from a map in  $\TT$.
The $B^*_t$ invariance of $g_0$ was
the reason for choosing $\Phi = A \circ \phi$,
rather than $\phi \circ A$.

It follows that the translations $E_1$ and $E_2$ are isometries w.r.t. the pull-back metrices on each level $\RR^2\times \{t\}$. 
\end{enumerate}

\end{observation}
 
 \begin{definition}
The description of the metric $\OG_1$ on $\RR^2\times[0,1]$ as in \ref{metric} together with Observation \ref{iso} defines a metric on $M$.
We denote this metric on $M$ by $G_1$.
\end{definition}

 Restrict the metric $G_1$ on $M_1=M\setminus N(\C)$.

Before describing the metric on the whole manifold $\N$, we prove the following
lemma, whose proof is simple but crucial result derived from the above defined Riemannian metric on $(\M_1,\mathcal{G}_1)$

\begin{remark}
We start with the flow $\psi_1$ in $M_1$ or $M$.
In the following proof and in other situations in this article
we analyze flow lines in the lift $\OM_1$. This is contained
in the universal cover $\widetilde{M}$. Some arguments are done
in $\widetilde{M}$. For simplicity of notation we denote
the lift of $\psi_t$ to $\OM_1$ and the lift of $\psi_t$ to
$\widetilde{M}$ by the same notation $\widehat{\psi}_t$.
\end{remark}

\begin{lemma}\label{qgo} 
If $\gamma$ is a flow line or flow ray   of $\psi_t$ contained  $M_1$  then the lift of $\gamma$ in the universal cover $\widetilde{M}_1$ is a globally length minimizing geodesic.

\end{lemma}
\begin{proof}
Consider the universal cover $\widetilde{M}=\RR^2\times\RR$ 
with the metric $\OG_1$
and the subset $\OM_1$ contained in it.
It is enough to prove that the flow lines or flow rays in $\OM_1$
are length minimizing in $\OM_1$ w.r.t. $\OG_1$.
This is because any rectifiable curve $\gamma$ in $\widetilde{M}_1$ 
joining
two points in a flow line or flow ray in $\widetilde{M}_1$ projects
to a curve in $\OM_1$ joining two points in a flow line or flow ray.

What we prove is that the flow lines in $\widetilde{M}$ are
globally length minimizing. This implies that that
the flow lines or flow rays contained
in  $\OM_1$ are length minimizing in $\OM_1$.
Notice also that we denote the flows in $\widetilde{M}$ or $\OM_1$
by $\widehat{\psi}_t$, see previous remark.

We note that the flow lines of the suspension flow $\widehat{\psi}_t$ in 
$\widetilde{M}=\RR^2\times\RR$ are the vertical lines $\{*\}\times\RR$ in $\RR^2\times\RR$.
These lines are the integral curves of the vector field 
$\mathscr{T}=\frac{d}{dt}\widehat{\psi}_t$ in $\RR^2\times\RR$. As the vectors in the $t$-directions are orthogonal to the vectors in the span of $\{\mathscr{X},\mathscr{Y}\}$ (i,e. in the horizontal levels $\RR^2\times \{t\}$), the integral curves of the vector field $\mathscr{T}=\frac{d}{dt}\widehat{\psi}_t$ are globally length minimizing. More precisely, consider a vertical line $\{a\}\times\RR$ and take two points $p_1$ and $p_2$ on it. Suppose $\sigma$ is a curve connecting $p_1$ and $p_2$ then

\begin{equation*} 
\begin{split}
\text{length}(\sigma) & =\int_{\text{domain}(\sigma)}||\frac{d}{dt}\gamma(t)||dt \\
 & =\int_{\text{domain}(\sigma)}(||\frac{d}{dt}\gamma(t)|_{\mathscr{T}(\gamma(t))}||^2+||\frac{d}{dt}\gamma(t)|_{\mathscr{X},\mathscr{Y}(\gamma(t))}||^2)^{1/2}dt\\
 & \geq \int_{\text{domain}(\sigma)}||\frac{d}{dt}\gamma(t)|_{\mathscr{T}(\gamma(t))}||dt
\end{split}
\end{equation*}
Where $\frac{d}{dt}\gamma(t)|_{\mathscr{T}(\gamma(t))}$ 
and $\frac{d}{dt}\gamma(t)|_{\mathscr{X},\mathscr{Y}(\gamma(t))}$ 
denote the components of 
$\gamma'(t)=\frac{d}{dt}\gamma(t)$ along $t$-direction 
and in the span of $\{\mathscr{X},\mathscr{Y}\}$ respectively. 
Hence it is clear that the integral curves of the vector field 
$\mathscr{T}=\frac{d}{dt}\widehat{\psi}_t$, i,e the flow lines 
of $\widehat{\psi}_t$ in $\RR^2\times\RR$, are globally 
length minimizing geodesic in 
$\widehat{M}=\RR^2\times\RR$ w.r.t. the metric $\OG_1$. 

As the flow lines of $\widehat{\psi}_t$ in $\MM$ are length minimizing geodesic, so is the flow lines or flow rays of the restricted flow $\widehat{\psi}_t|_{\OM_1}$ on $\OM_1$ inside $\MM$. 

Note that $\widehat{M}_1$
 is not the universal cover of $M_1$, but an infinite subcover of $M_1$. As the flow rays are globally length minimizing geodesic in a subcover, same has to be true in the universal cover $\widetilde{M}_1$. 
This completes the proof. 
\end{proof}
We end this subsection with a crucial remark which will be used later. 
\begin{remark}\label{solv}

     The construction of our metric $G_1$ is motivated by the $Solv$ metric $dS^2=\lambda^{-2s}dx^2+\lambda^{2s}dy^2+ds^2$ on $\RR^2\times\RR$. We observe that there is an another way to see the metric $\OG_1$ on $M_1$, we can find a map $\mathcal{H}$ such that the family of map $h_s$ can be written as $h_s=\overline{B}_s\circ \mathcal{H}$ where $\overline{B}_s(x,y,s)=(\lambda^{-s}x,\lambda^s y,0)$. To see that such an $\mathcal{H}$ exits it is enough to determine the map near the $t$-axis and it is easy to check that the map $\mathcal{H}(x,y,s)=((\tilde{\theta}(x,y))^s x,y,s)$ near $\{(0,0,t)|t\in\RR\}$ serves the purpose. 
We can extend it to all of $\widetilde{M}$ using the 
definition of $\theta$ in $\TT$. 
Hence the pull-back metric on the level $\RR^2\times\{s\}$ is same with $\mathcal{H}^*( \overline{B}_s^*(g_0))$.

If we consider the family of pull-back metrices $\overline{B}_s^*(g_0)$ where $g_0^2=dx^2+dy^2$ on the level $\RR^2\times\{0\}$, it is easy to check that we get the \textit{Solv} metric $dS$ on $\MM=\RR^2\times\RR$,
$$dS^2=\lambda^{-2s}dx^2+\lambda^{2s}dy^2+dt^2\text{ for }s\in[0,1]$$
Hence the metric we defined, $\OG_1=\mathcal{H}^*(dS)$, is a bounded perturbation of the \textit{Solv} metric on each level $\RR^2\times\{s\}$.

\end{remark}
\subsection{Defining the metric on the whole manifold:} 
As before we consider a collection of attracting and repelling Franks-Williams type hyperbolic plugs 
$$\{(\M_1,\Psi_t^1);(\M_2,\Psi_t^2);(\M_3,\Psi_t^3);...;(\M_n,\Psi_t^n\}$$

By our construction, we have a metric $G_i$ on $M_i\subset \M_i$ for each $i$. 
The (possibly disconnected)
surface $\partial M_1$ is smooth in the metric in $M$
and hence inherits a Riemannian metric. 
With a little bit of care we can assume without loss of generality
that all boundary components of all $\partial M_i$ are
pairwise isometric.
Now we smoothly extend the metrics $G_i$ on the union collar neighbourhoods $\cup_i \partial M_i\times [0,1]$ to get a Riemannian metric $\GG$ on the whole manifold $\NN$, in particular, $\GG|_{M_i}=G_i$ for all $i$. 
This metric in $\partial M_i \times [0,1]$ is a product
metric, so that projection onto the first factor is a 
length decreasing map.



\begin{remark} The manifolds $\M_i$ have boundary
which is $\pi_1$-injective in $\N$ for any $i$, because 
each boundary component of each $\M_i$ is $\pi_1$-injective
in $\M_i$. 
It follows that any lift of $\M_i$ to the universal
cover $\NN$ is itself a universal cover of $\M_i$. 
So we can think of these lifts as copies $\MM_i$ of
the universal cover of $\M_i$ which are contained
in $\NN$.
\end{remark}

The restriction of $\GG$ to a lift $\MM_i$ of a
single hyperbolic plug  is denoted by $\GG_i$, and the path metric induced by $\GG_i$ on $\MM_i$ is denoted by $d_{\GG_i}$.

\vskip .07in
\textbf{Notation:}\begin{itemize}
    \item For any $\M_i$, $\G|_{\M_i} =\G_i$
    \item for any two point $p_1$ and $p_2$ in $\MM_i$,
    $$d_{\GG_i}(p,q)=\text{minimum}\{\text{length}_{\GG_i}(\sigma)|\sigma\text{ is a curve contained in }\MM_i\text{ connecting }p_1,p_2\}$$
\end{itemize}
 

We finish this section showing that the flow lines or flow rays contained in the lift $\MM_i$ of a single hyperbolic plug in the universal cover $\NN$ is a quasigeodesic w.r.t. the restricted metric $d_{\GG_i}$. 

\begin{lemma}\label{qgo1}
There is an $\epsilon>0$ such that if $\gamma$ is a flow ray or flow line fully contained in some $\MM_i$, and $p_1$, $p_2$ are two points on $\gamma\subset\MM_i$, then
$$\text{length}_{\GG_i}(\gamma_{[p_1,p_2]})\leq d_{\GG_i}(p_1,p_2)+\epsilon$$
\end{lemma}

\begin{proof}
W.l.o.g. we prove the lemma in $\MM_1$. Note that $\widetilde{M}_1\subset \MM_1$ and if $p_1,p_2\in\widetilde{M}_1$, then the result is true because inside $\widetilde{M}_1$, flow segments are length minimizing in $\widetilde{M_1}$
by Lemma \ref{qgo}. 
Furthermore by construction of the metric in 
$\partial M_1 \times [0,1]$, there
is a length decreasing retraction $\M_1$ to $M_1$,
so any minimal path in $\M_1$ between points
in $M_1$ is actually contained in $M_1$.
It follows that the
result works if $p_1, p_2 \in \widetilde{M}_1$.

Hence we assume that $p_2\in \MM_1\setminus \widetilde{M}_1$,
and $p_1 \in \widetilde{M}_1$.

$\MM_1$ can be written as $\widetilde{M_1}\cup(\partial \widetilde{M}_1\times[0,1])$. By the compactness of $\partial M_1\times[0,1]$ we can consider $\epsilon>0$ such that for any flow ray $\gamma$,
$$\text{length}_{\GG_1}(\gamma\cap (\partial \widetilde{M}_1\times [0,1]))\leq \epsilon/2$$
As there are only finitely many plugs, we can choose $\epsilon$ big enough such that it works for all $\partial \widetilde{M}_i\times [0,1]$. 

Now consider the flow segment $\gamma_{[p_1,p_2]}$, as $p_2\in\MM_1\setminus \widetilde{M}_1$, the flow segment must intersect $\partial \widetilde{M}_1$ at a single point, say $p_3$. Then
\begin{equation*}
    \begin{split}
        \text{length}_{\GG_1}(\gamma_{[p_1,p_2]})&=\text{length}_{\GG_1}(\gamma_{[p_1,p_3]})+\text{length}_{\GG_1}(\gamma_{[p_3,p_2]})\\
        					 &=d_{\GG_1}(\gamma_{[p_1,p_3]})+\text{length}_{\GG_1}(\gamma_{[p_3,p_2]})\\
                                                 &\leq d_{\GG_i}(p_1,p_3)+\epsilon/2\\
                                                 &\leq d_{\GG_i}(p_1,p_2)+d(p_2,p_3)+\epsilon/2\\
                                                 &\leq d_{\GG_i}(p_1,p_2)+\epsilon/2+\epsilon/2
    \end{split}
\end{equation*}
\end{proof}

\section{Analysis of the Flowlines}\label{S5}
In this section we show that every flow line of the flow $\widetilde{\Psi}_t$ on $\NN$ is quasigeodesic with respect to the path metric induced by Riemannian metric $\GG$ as constructed in the previous section.
As before, $\N$ is made of the collection of hyperbolic plugs 
$$\{(\M_1, \Psi_t^1, \mathcal{G}_1);(\M_2, \Psi_t^2, \mathcal{G}_2);(\M_3, \Psi_t^3, \mathcal{G}_3);...;(\M_n, \Psi_t^n, \mathcal{G}_n)\}$$
First we consider a single hyperbolic plug, say $(\M_1,\Psi^1_t, \G_1)$, and analyze the properties of the flow lines of the semi-flow $\Psi^1_t$.

\subsection{Flowlines in the plug $(\M_1,\G_1,\Psi^{1}_t)$}\label{plug}

Recall that $\M_1$ is made from $M=\frac{\TT\times \RR}{(q,t)\sim(\Phi(q),t-1)}$ as in Subsection \ref{Const}, where $\Phi=A \circ \phi$ is a repelling DA map. The universal cover $\widetilde{M}$ can be considered as $\RR^2\times\RR$ equipped with the coordinate system $\{x,y,t\}$ where $x$-directions and $y$-directions are parallel to the strong stable and the strong unstable directions of the hyperbolic map $\widetilde{A}$ on $\RR^2$ and $t$-directions are along the flow lines of the suspension flow $\widehat{\psi}^1_t$. 

Consider the repelling periodic orbit $\C$ of $\psi_t$ in $M$ and $N(\C)$ is an open solid torus neighbourhood of $\C$ as described in subsection 2.1. Let $\widehat{N(\C)}$ be the collection of lifts of $N(\C)$ in $\RR^2\times \RR$, then $\OM_1=(\RR^2\times \RR)\setminus \widehat{N(\C)}$ is an infinite cover of $M_1=M\setminus N(\C)$, but not the universal cover.

$M_1$ is equipped with the Riemannian metric $G_1$ as constructed in the previous section, similarly, we denote the lifted metric on $\OM_1$ by $\OG_1$. 

$\OM_1$ is a manifold with boundary, where $\partial\OM_1$ is the lift of $\partial{M_1}=T_1$. $\partial\OM_1$ is a collection of disjoint 
infinite cylinders in $\RR^2\times \RR$ which are transverse to the flow lines. Suppose $\partial \OM_1=\cup_{i\in\mathbb{N}}\OOT_i$, where $\OOT_i$'s are the infinite cylindrical boundary components of $\partial \OM_1$. 

There are exactly two types of flow lines of the lifted  semiflow $\widehat{\psi}^1_t|_{\OM_1}$ in $\OM_1$, if $\gamma$ is a flow line in $\OM_1$, then
\begin{enumerate}
     \item either $\gamma$ is contained in $\OA$, the lift of the attractor $\A\subset M_1$, or
    \item $\gamma$ intersects a lift of $\partial M_1=T_1$. Moreover $\gamma$ intersects exactly one such lift of $T_1$. 
\end{enumerate}

\begin{remark}
When we are dealing only with $M_1, \M_1$
we will simplify notation and denote $\psi^1_t$ and $\Psi^1_t$
by $\psi_t$ and $\Psi_t$.
\end{remark}

 In this subsection we show that almost all flow rays which intersect the boundary $\partial\OM_1$ `go away' from the boundary component it intersects in an efficient manner as time $t$ goes to positive infinity. In the next subsection, we extend the result in the universal cover $\MM_1$. To state the precise statement we first need to fix some notations.

		\begin{figure}[h!]
  \centering
  \begin{subfigure}[b]{0.48\linewidth}
    \includegraphics[width=\linewidth]{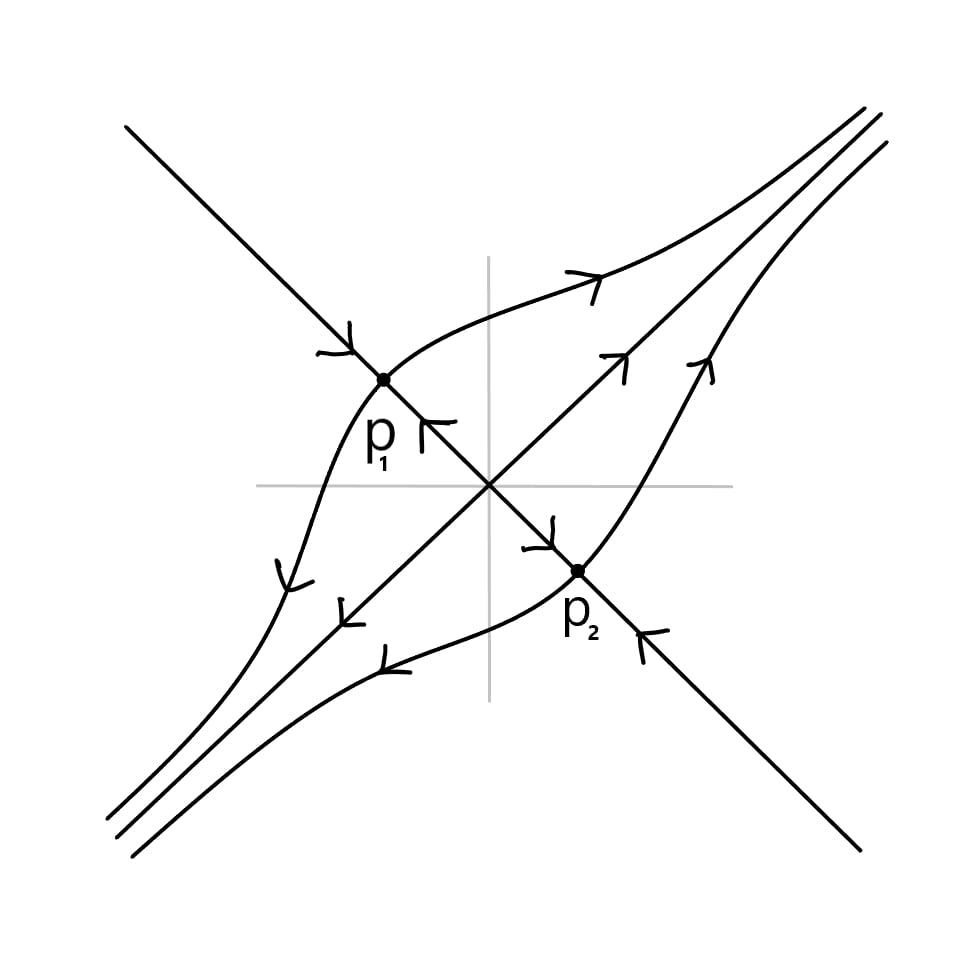}
     \caption{ }
  \end{subfigure}
  \hskip .05in
  \begin{subfigure}[b]{0.5\linewidth}
    \includegraphics[width=\linewidth]{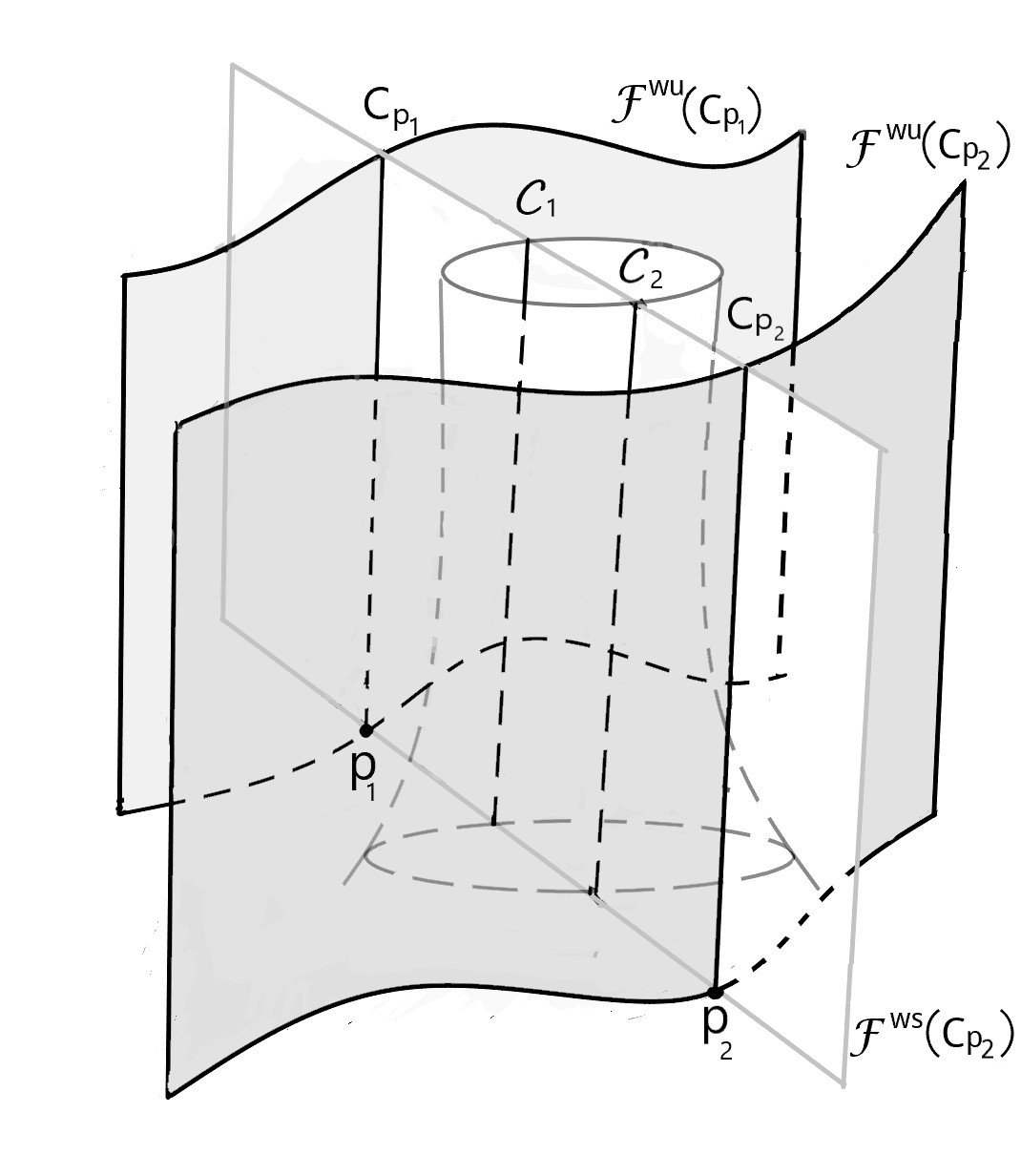}
     \caption{ }
  \end{subfigure}
    \caption{(A) The two dimensional image of the blow up of
a hyperbolic point,
  (B) The three dimensional image of the blow up
of a hyperbolic orbit.}
  \label{fig1}
\end{figure}
		
Consider the repelling fixed point $\mathfrak{o}$ of the DA map $\Phi=A \circ\phi$ on $\TT$. There are also two hyperbolic fixed points, denoted by
$p_1$ and $p_2$, on the attractor in $\TT$ as shown in Figure \ref{fig1} (A). In the suspension manifold $M=\TT\times[0,1]/\sim$ there are two periodic orbits, $C_{p_1}$ and $C_{p_2}$, coming from $p_1$ and $p_2$ and these two orbits are contained the attractor $\A$ of $\psi^1_t$. Now we consider the repelling orbit $\C$ and the open solid torus neighbourhood $N(\C)$ around it. The weak stable leaves of $C_{p_1}$ and $C_{p_2}$ intersect the boundary torus $T_1=\partial N(\C)=\partial M_1$ in two circles, denoted by
$\C_1$ and $\C_2$, as shown in Figure \ref{fig1} (B). 
\begin{definition}
Suppose $q\in \OM_1$ is such that the flow line $\gamma_q$ through $q$ intersects $\partial\OM_1$ at the boundary component denoted by $\widehat{T}_q$. We know that the flow ray $\gamma_q$ can intersect exactly one such component of $\partial \OM_1$. Then let
\begin{itemize}

    \item $\D_{\OG_1}(q,\widehat{T}_q)$ ( resp. $\D_{S}(q,\widehat{T}_q)$) denote the distance between $q$ and $\widehat{T}_q$ w.r.t. the path metric induced by $\OG_1$ ( resp. the Solv metric $dS$ ) on $\OM_1$. 

    \item $\ell_{\OG_1}(q)$( resp. $\ell_{S}(q)$) is the length of the flow line segment connecting $q$ and $\widehat{T}_q$ w.r.t. $\OG_1$( resp. $dS$).

\end{itemize}
\end{definition}
Next we fix a small number $\delta>0$ and consider the open $\delta$-neighbourhood $N_{\delta}(\C_1\cup \C_2)$ on $T_1$ and let $N_{\delta}(\widehat{\C}_1\cup \widehat{\C}_2)$ denote the lift of
$N_{\delta}(\C_1\cup \C_2)$ in $\OM_1$. Notice that 
$\CC_1\cup\CC_2$ is a countable, infinite
 collection of propertly embedded
lines in $\OM_1$. 
The following is a key technical result used in this article.
\begin{proposition}\label{key}
There exists $C > 1$ and $c > 0$ satisfying the following:
Let $q\in \OM_1$ such that the flow ray through $q$  which intersects $\partial\OM_1$, say at the boundary component $\widehat{T}_q$. If $\gamma_q$ intersects $\widehat{T}_q$ on the region $\widehat{T}_q\setminus N_{\delta}(\CC_1\cup\CC_2)$, 
then
$$\ell_{\OG_1}(q)\leq C\D_{\OG_1}(q, \widehat{T}_q) +c$$
\end{proposition}

\begin{proof} We will prove the lemma in the manifold $\widetilde{M}=\RR^2\times\RR$ using the coordinate system $\{x,y,t\}$ on it.
This is enough, because we have $\OM_1\subset \widetilde{M}=\RR^2\times\RR$ and distance between two points in $\OM_1$ w.r.t. the metric $d_{\OG_1}$ is bigger than 
distance in $\RR^2\times\RR$ w.r.t. $d_{\OG}$. 
In addition length of a flow segment is the same in both $\OM_1$ and
 $\widetilde{M} = \RR^2\times\RR$, \  as the Riemannian metric 
in $\OM_1$ is the one induced from the inclusion
$\OM_1 \subset \widetilde{M} = \RR^2\times\RR$.

It is enough to prove the result for one  component of the boundary $\partial \OM_1$ as we can permute the boundary components using the translation isometries. In addition up to changing the coordinates $(x,y)$ we can
assume the orbit passes through $(0,0,0)$.
In other words
we fix the lift of the repelling orbit $\C$ passing through $(0,0,0)\in \RR^2\times\RR$, i,e the line $\{(0,0,t)|t\in\RR\}$ and the lift of the $\partial M_1=T_1$ around this line, which we denote by $\OOT_0$.

We will first prove the result for the points $q$ on the $yt$-plane $\{x=0\}$. The reason behind it is that the Riemannian metric on on the $yt$-plane is almost $\lambda^{2t}dy^2+dt^2$. We make this more precise in the following remark:

\begin{remark}\label{sol}
\begin{enumerate}

  \item Note that the map $\tilde{\phi}:\RR^2\rightarrow\RR^2$ perturbs the $y$-directions only near the lifts in $\RR^2$ of the repelling fixed point $\frak{o}$, and outside those neighbourhoods the map $\widetilde{\phi}$ is the identity map. Hence it is also true that that map $\mathcal{H}$ distorts the $yt$-plane boundedly in $\RR^2\times\RR$. 
By Remark \ref{solv}, we know that $\OG_1=\mathcal{H}^{*}(dS)$ where $dS$ is the $Solv$ metric.
In particular, the $Solv$ metric restricted on the $yt$-plane is $\lambda^{2t}dy^2+dt^2$, and the $x$-directions are everywhere perpendicular to the $yt$-plane. Hence we can assume that the Riemannian metric on the $yt$-plane induced by $\OG_1$ is boundedly distorted from and very close to
the $Solv$ metric. In other words, we can find two constants $a_0>1$ and $a_1>0$ such that if $\sigma$ is a curve on the $yt$-plane then

$$\frac{1}{a_0}\text{length}_{\OG_1}(\sigma)-a_1\leq\text{length}_{S}(\sigma)\leq a_0\text{ length}_{\OG_1}(\sigma)+a_1$$

where $\text{length}_S(\sigma)$ and $\text{length}_{\OG_1}(\sigma)$ means length of $\sigma$ w.r.t. the  Solv metric $dS$ and $\OG_1$ respectively.

\item Note that the $yt$-plane is the unstable leaf through $(0,0,0)$ for the suspension of the linear Anosov map $\tilde{A}$.
Notice that the $yt$-plane is not invariant under the map $\widetilde{A}\circ \widetilde{\phi}$.

\end{enumerate}
\end{remark}
 Now we can state the precise result we want to prove on the $yt$-plane:

\begin{lemma}\label{unstable}
There exists $K>1$ and $k>0$ such that for any $q$ on the intersection
of the $yt$-plane with $\OM_1$ for which the flow line $\gamma_q$ intersects $\OT$, we have 
$$\ell_{\OG_1}(q)\leq K\D_{\OG_1}(q, \OT)+k$$
\end{lemma}

\begin{proof}

Consider a point $q$ on the $yt$-plane such that the flow ray $\gamma_q$ passing through $q$ intersects $\OOT_0$. 

As the manifolds are complete every distance can be realized by a curve. Suppose $\sigma:[0,1]\rightarrow \widetilde{M}$ denotes a shortest path connecting $q$ and $\OOT_0$ in $\OM_1$.
\begin{claim}
Any shortest path $\sigma$ between $q$ and $\OOT_0$ lies on $yt$-plane. 
\end{claim}
\begin{proof}
Suppose $\sigma$ is parametrized as $\sigma(t)=(\sigma_1(t),\sigma_2(t),\sigma_3(t))$ in $\RR^2\times \RR$. Consider the projection of $\sigma$ on the $yt$-plane, $\Pi(\sigma(t))=(0,\sigma_2(t),\sigma_3(t))$. Suppose $\Pi(\sigma(a))$ is the first intersection point of $\Pi(\sigma)$ and $\OT$. We show that 
$$\text{length}(\Pi(\sigma))|_{[0,a]}\leq \text{length}(\sigma)$$
and the claim follows from that. 

In Remark \ref{solv} we have seen that the metric $\OG_1$ can be expressed as the pull-back of the $Solv$ metric $dS$ by the map $\mathcal{H}$,
i,e $\widehat{G}_1=\mathcal{H}^*(dS)$, where $dS=\lambda^{-2s}dx^2+\lambda^{2s}dy^2+dt^2$

Consider the map $\mathcal{H}:\RR^2\times\RR\rightarrow \RR^2\times\RR$ as in Remark \ref{solv} . As $\OG_1=\mathcal{H}^*(dS)$, 
for any curve $\zeta$,
$$\text{length}(\zeta)\text{ w.r.t. the metric }\OG_1=\text{length}(\mathcal{H}(\zeta))\text{ w.r.t. the metric }dS$$

As the directions $x$, $y$ and $t$ are orthogonal to each other w.r.t. the \textit{Solv} metric $dS$, it is easy to check that for any curve $\zeta$ in $\RR^2\times\RR$ and its projection $\Pi(\zeta)$ on $yt$-plane, we have
$$\text{length}(\Pi(\delta))\leq \text{length}(\sigma)\text{ w.r.t. the metric }dS $$

Moreover, the projection map $\Pi$ commutes with $\mathcal{H}$ because $\mathcal{H}$ does not change the $y$-coordinate of a point and $\Pi$ is projection on the $yt$-plane.
This can also be easily verified by the formulas.

Hence we conclude  
\begin{align*}
   & \text{length}(\Pi(\mathcal{H}(\sigma)))\leq \text{length}(\mathcal{H}(\sigma))\text{ w.r.t. the metric }dS \\
    \implies &\text{length}(\mathcal{H}(\Pi(\sigma)))\leq \text{length}(\mathcal{H}(\sigma))\text{ w.r.t. the metric }dS \\
    \implies &\text{length}(\Pi(\sigma)))\leq \text{length}(\sigma)\text{ w.r.t. the metric }\OG_1
\end{align*}
This completes the claim.
\end{proof}

 
 In the rest of the proof of Lemma
\ref{unstable} we will use the $Solv$ metric on $yt$-plane, i,e the metric $\lambda^{2t}dy^2+dt^2$ for simplicity of calculations. In Remark \ref{iso}
we have seen that the $Solv$ metric on the $yt$-plane is quasi-isometric to the metric induced by $\OG_1$. 
 

\vskip 0.1in

Consider the point $q=(0,c,0)$ with $c>0$ on $\OT\cap yt$-plane and the forward flow ray $\gamma_{q}=\widetilde{\psi}^1_{[0,\infty)}(p)$ as shown in Figure \ref{fig2}, $\gamma_q=\{(0,c,t)|t\in[0,\infty)\}$. 

We prove the lemma for the flow ray $\gamma_q$ first. Later we explain how to 
derive the result for all other flow lines in the $yt$ plane intersecting
$\OOT_0$
from this
particular result on $\gamma_q$.
\vskip 0.1in

\noindent
\textbf{Proof for the flow line $\gamma_q$:} Suppose $q'=(0,c,t')$ is a point on $\gamma_q$ and let $\sigma_{q'}$ is a length minimizing curve on the $yt$-plane joining $q'$ and $\OT$. 
We define three functions as follows:
\begin{figure}
		\centering
			
			\includegraphics[width=1\linewidth]{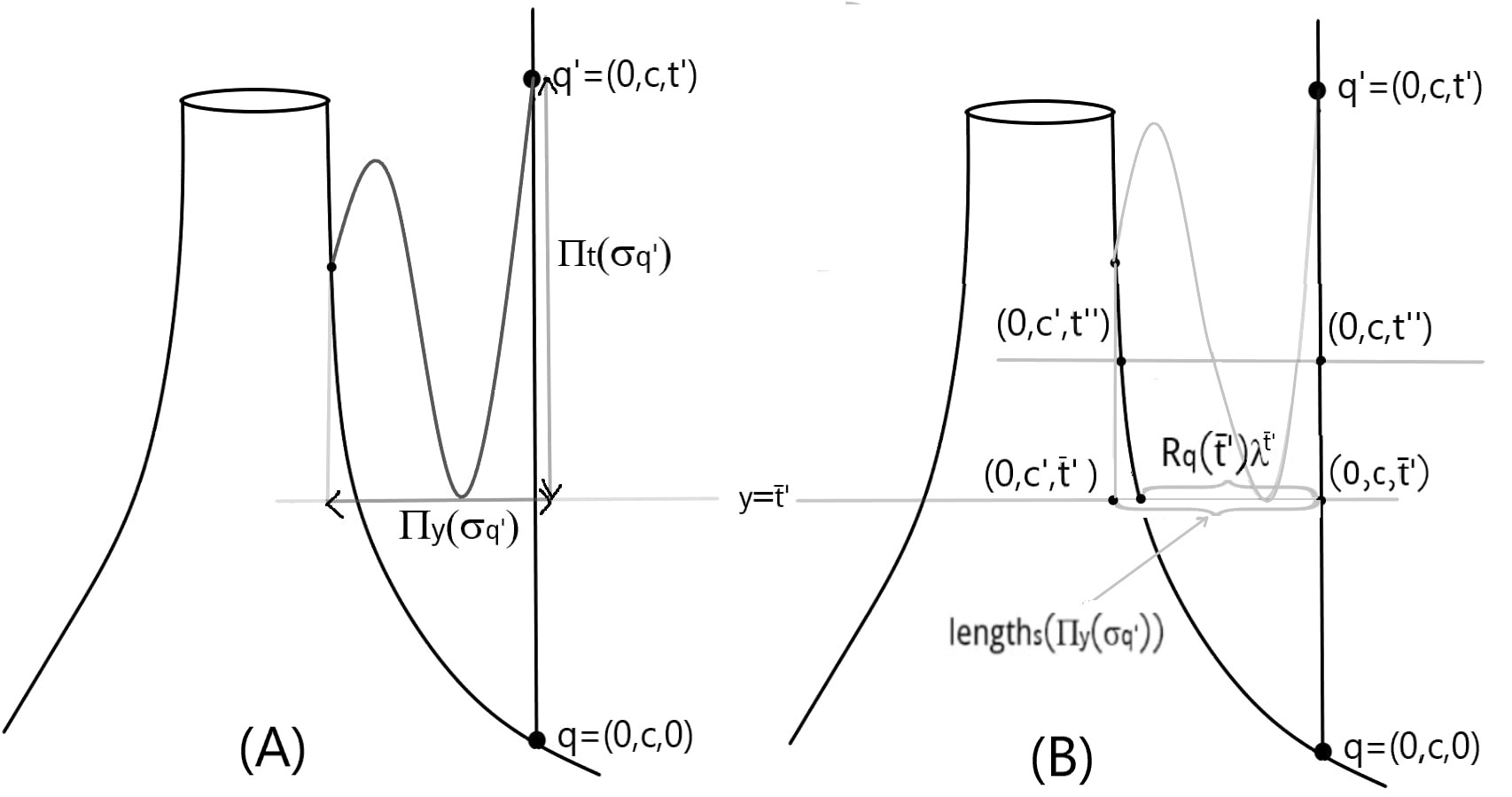}
			\caption{(A) This shows  a minimal
path in the $yt$-plane from $q'$ to the lift of the boundary 
torus, (B) This figure shows several geometric quantities that
are used in the analysis. In particular 
$R_q(\overline{t}') \lambda^{\overline{t}'}$ is the length in the
solvable metric of the horizontal segment depicted at height
$\overline{t}'$.}
			
			\label{fig2}
		\end{figure}

		
\begin{enumerate}
    \item Let $\Pi_t(\sigma_{q'})$ denotes the projection of $\sigma_{q'}$ on the vertical line $\gamma_q=\{(0,c,t)|t\in\RR\}$ along $y$-directions
$-$ same as horizontal directions
(notice that $\sigma_{q'}$ is contained in the $yt$ plane).
    \item Suppose $\bar{t'}$ is the lowest $t$-value attained by the curve $\sigma_{q'}$ as shown in the Figure \ref{fig2}(A). Then we denote the projection of $\sigma_{q'}$ on the line $t=\bar{t'}$ along $t$-directions $-$ same
as vertical directions $-$ by $\Pi_y(\sigma_{q'})$.
    \item For a point $(0,c,t'')$ on $\gamma_{q'}$, suppose the line $t=t''$ intersects $\OT$ in the positive side of $y$-direction at $(0,c',t'')$ as shown in the Figure \ref{fig2}(B).  Then we define $R_{q}(t'')=|c-c'|$. Note that, w.r.t. the $Solv$ metric $dS$ the length of the segment on the line $t=t''$ connecting $(0,c,t'')$ and $(0,c',t'')$ is $\lambda^{t''}|c-c'|$ or $R_{q}(t'')\lambda^{t''}$. Clearly $R_{q}$ is an increasing function of $t$ on $[0,\infty)$.

\end{enumerate}
Now we are ready to state two lower estimates of $\D_{S}(q',\OT)$.
\begin{claim}\label{c0}
Suppose $\bar{t'}$ is the lowest $t$-value in the projection $\Pi_t(\sigma_{q'})$ as shown in Figure \ref{fig2}(a). Then 
\begin{enumerate}
    \item $t'-\bar{t'}\leq \D_{S}(q',\OT)$
    
    \item In addition, $R_{q}(\bar{t'})\lambda^{\bar{t'}}\leq \D_{S}(q',\OT)$
\end{enumerate}

\end{claim}
\begin{proof}
As the $t$-directions and $y$-directions are everywhere orthogonal w.r.t the metric $\lambda^{2t}dy^2+dt^2$, it is easy to check that $\text{length}_{S}(\Pi_t(\sigma_{q'}))\leq \text{length}_{S}(\sigma_{q'})$ and  in addition
we have 
$\text{length}_S(\Pi_y(\sigma_{q'}))\leq \text{length}_S(\sigma_{q'})=\D_S(q',\OOT_0)$. 
\begin{enumerate}
    \item 
As the curve $\Pi_t(\sigma_{q'})$ 
connects the points $q'=(0,c,t')$ and $(0,c,\bar{t'})$ it is clear that $|t'-\bar{t'}|\leq \text{length}_S(\Pi_t(\sigma_{q'}))\leq\D_{S}(q',\OT)$.

 Note that $\bar{t'}$ can not be a negative number, because if $\bar{t'}<0$, then 
$$t'-\bar{t'}>t'=\ell_S(q')\geq \D_{S}(q',\OT)$$
which can not be true as we have just proved $t'-\bar{t'}\leq \D_{S}(q',\OT)$.
\item
As $\bar{t'}\geq 0$, we observe that the line segment
 $\{(0,y,\bar{t'})|y\in[c-R_{q}(\bar{t'}),c]\}$ is contained in the curve $\Pi_y(\sigma_{q'})$ as shown in Figure \ref{fig2}(b). Hence
$$R_{q}(\bar{t'})\lambda^{\bar{t'}}\leq\text{length}_S(\Pi_t(\sigma_{q'}))\leq\text{length}_S(\sigma_{q'}) =\D_{S}(q',\OT)$$
\end{enumerate}
This proves the claim.
\end{proof}
Consider the function
$P_{q}(t)=\frac{R_{q}(\frac{t}{2})\lambda^{\frac{t}{2}}}{t}$ on $t\in(0,\infty)$.

As the function $R_{q}(t/2)$ is increasing and $\lambda^{t/2}/t$ is strictly increasing for large $t$ values, we can fix a value $k'>0$ such that 
$P_{q}(t)>1$ when $t>k'$.

For the point $q'=(0,c,t')$, if $t'>k'$ then $R_q(t'/2)\lambda^{t'/2}>t'$. On the other hand $t'=\ell_{S}(q')\geq \D_{S}(q',\OOT_0)$. Combining these two inequalities we get,
\begin{equation}\label{e1}
    \D_{S}(q',\OT)\leq \ell_{S}(q')=t' < R_{q}(t'/2)\lambda^{t'/2} 
\end{equation}
Hence by the Claim item \ref{c0}(2), \  $t'/2$ can not be smaller than the lowest $t$-value attained by the curve $\Pi_t(\sigma_{q'})$, i,e  $t'/2>\bar{t'}$. Otherwise we get
$$R_{q}(t'/2)\lambda^{t'/2}\leq R_{q}(\bar{t'})\lambda^{\bar{t'}}\leq \D_{S}(q',\OT)$$
which contradicts Equation \ref{e1}. 

Finally, as $t'/2>\bar{t'}$, applying Claim \ref{c0}(1) we deduce
\begin{equation}\label{e2}
\begin{split}
   & t'-t'/2\leq t'-\bar{t'}\leq \D_{S}(q',\OT)\\
   \implies & t'/2\leq \D_{S}(q',\OT) \\
   \implies & t'\leq 2\D_{S}(q',\OT)\\
   \implies & \boxed{\ell_S(q')\leq 2\D_{S}(q',\OT)\text{ when }q'=(0,c,t')\text{ and }t'>k'}
\end{split}
\end{equation}

If $t'\leq k'$, then for a point $q'=(0,c,t')$ we get 
\begin{equation}\label{e3}
    \ell_{S}(q')\leq k'\leq k'+2\D_{S}(q',\OT)
\end{equation}
Combining \ref{e2} and \ref{e3}, we conclude
\begin{equation}\label{e4}
\ell_{S}(q')\leq2\D_{S}(q',\OT)+k'\text{ for all } q'\in\gamma_q=\{(0,c,t)|t\in[0,\infty)\}
\end{equation}

Finally, by using Remark \ref{sol}, we can deduce $\D_{S}(q',\OT)\leq a_0\D_{\OG}(q',\OT)+a_1$ and replacing it in the previous equation,

$$\ell_{S}(q')\leq2\D_{S}(q',\OT)+k'\leq 2a_0\D_{\OG_1}(q',\OT)+2a_1+k'\text{ for all } q'\in\gamma_q=\{(0,c,t)|t\in[0,\infty)\}$$

By renaming, $2a_0=K$ and $2a_1+k'=k$, and replacing $\ell_{S}(q')=\ell_{\OG_1}(q')$,
$$\boxed{\ell_{\OG_1}(q')\leq K\D_{\OG_1}(q',\OT)+k\text{ for all } q'\in\gamma_q=\{(0,c,t)|t\in[0,\infty)\}}$$

This completes the proof of Lemma \ref{unstable} only on the flow ray $\gamma_q$, where $q=(0,c,0)$. 
To deal with the other flow rays in the $yt$-plane we do the
following: consider a family of maps
$$\mu_a: \widetilde{M} \to \widetilde{M}, 
\ \ \mu_a(x,y,t) \ = \ (\lambda^a x, \lambda^{-a} y, t + a)$$

\noindent
where $a$ is an arbitrary real number.
It is easy to see that any $\mu_a$ is an isometry of
the solv metric $dS$. 
In addition, we choose the original torus $T_0$ transverse
to the flow so that $\mu_a$ leaves invariant 
the fixed lift $\OT$ of $T_0$ to $\mt$
for any $a \in \RR$. Notice that $\mu_a$ 
also fixes the $yt$ plane.
In the $yt$ plane, $\mu_a$ sends flow lines to flow lines.
If $p = \mu_a(q)$, then $p$ is also in $\OT$.
Now for any $p'$ in the forward flow line of $p$,
one obtains Equation \ref{e4} for $p'$ as well, 
since $\mu_a$ is an isometry of the solv metric.
This obtains all flow lines in the $yt$-plane, except
for the flow line through $(0,0,0)$, but this one
does not intersect $\OT$.
Now use Remark \ref{sol} to finish 
the proof of Lemma \ref{unstable}.
\end{proof}

		\begin{figure}[h!]
  \centering
  \begin{subfigure}[b]{0.5\linewidth}
    \includegraphics[width=\linewidth]{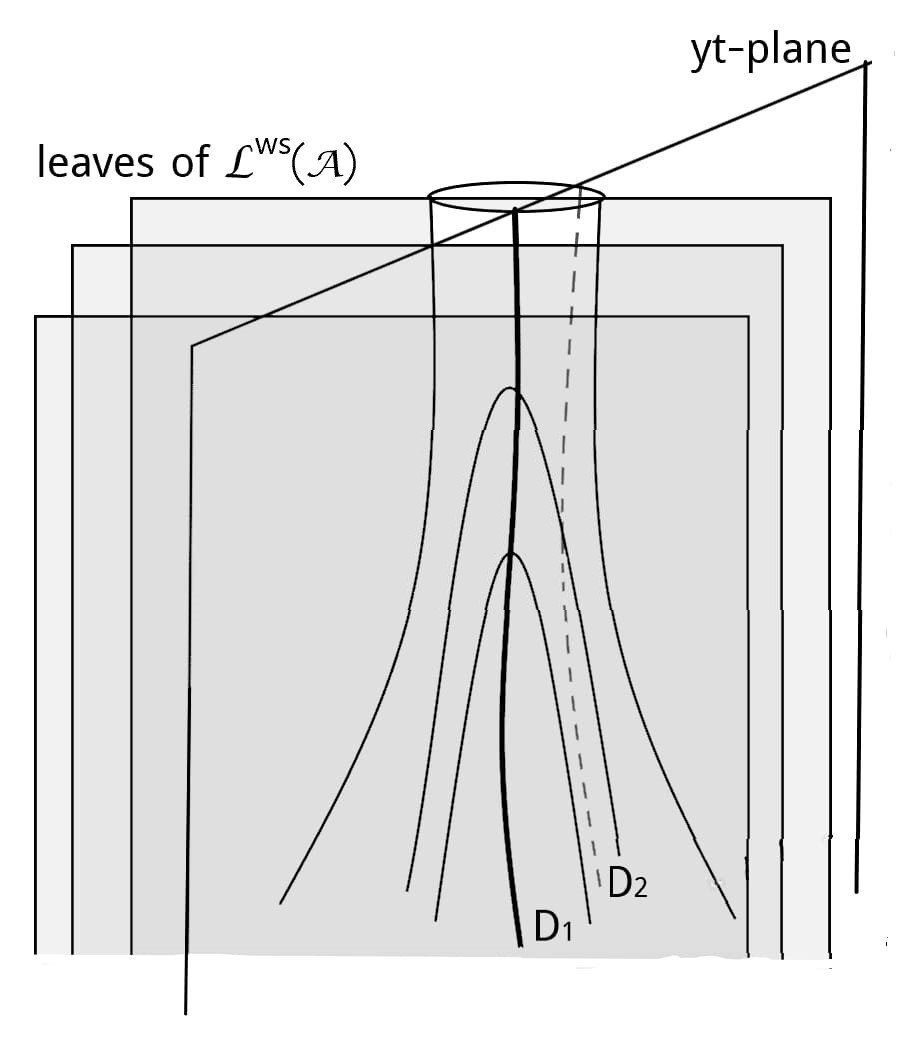}
  \caption{}
  \end{subfigure}
  \hskip 0.2in
   \begin{subfigure}[b]{0.3\linewidth}
    \includegraphics[width=\linewidth]{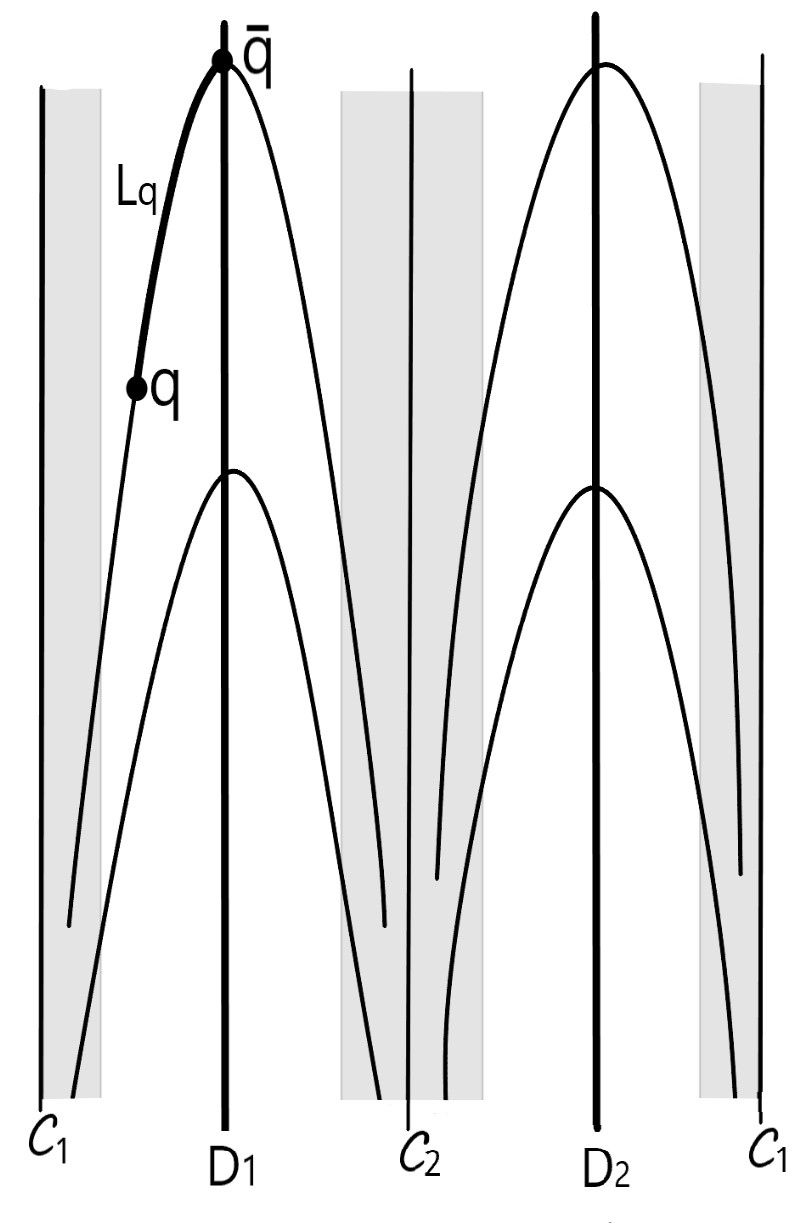}
    \caption{}
  \end{subfigure}
  \caption{(A) The foliation $\FF$ in a lift of
a torus; (B) Leaf of $\widetilde{\F}$
through $q$ intersect $D_1$ at $\bar{q}$. The shaded region represents the `bad region.}
  \label{fig4}
\end{figure}

Now we are ready to prove Proposition \ref{key}.
\vskip 0.2in
\noindent
\textbf{Proof of Proposition \ref{key}:} The semiflow $\psi^1_t$ has a hyperbolic set $\A$ in $M_1$, it is a two dimensional attractor. Consider the weak-stable foliation $\LE^{ws}$ of $\A$ associated with the semiflow
$\psi^1_t$ in $M_1$. Every point in
$M_1$ is in $\LE^{ws}$ as every point in $M_1$ is
attracted to $\A$. The foliation
 $\LE^{ws}$ intersects the boundary $\partial M_1$ in a one dimensional foliation which has two Reeb components \cite{FW80}, we denote this foliation on $\partial M_1$ by $\F$. As described before
the circular leaves of $\F$ are denoted by $\C_1$ and $\C_2$. They are the common boundary circles of the Reeb annuli of $\F$.

Suppose $\widehat{\mathcal{F}}$ is the lift of $\mathcal{F}$ on $\OT$. The intersection of the $yt$-plane and $\OT$ has two components, say $D_1$ and $D_2$ as shown in Figure \ref{fig4}. $\widehat{\mathcal{C}}_1$ and $\widehat{\mathcal{C}}_2$ are on two different sides of the $yt$-plane. As any leaf $L$ of $\widehat{\F}$ (except $\widehat{\mathcal{C}}_1$ and $\widehat{\mathcal{C}}_2$) is asymptotic to both $\widehat{\mathcal{C}}_2$ and $\widehat{\mathcal{C}}_2$, it must intersect the $yt$-plane in a single point either on $D_1$ or on $D_2$ as shown in Figure \ref{fig4}.

Consider $q\in\OT\setminus\{\widehat{\mathcal{C}}_1, \widehat{\mathcal{C}}_2\}$ and the leaf $L_q$ of $\widehat{\F}$ passing through $q$. As described above, $L_q$ intersects either $D_1$ or $D_2$ at a single point, w.l.o.g we assume that $L_q$ intersects $D_1$ and let $\q=D_1\cap L_q$ as shown in Figure \ref{fig4}. Note that, $q$ and $\q$ lie on the same leaf of the weak stable foliation of $\widehat{\A}$, hence there exists $s\in\RR$ such that $\widehat{\psi}^1_s(q)$ and $\q$ lie on the same strong stable leaf
of $\hat{\psi}^1_t$. It follows that
$$d_{\OG_1}(\widehat{\psi}^1_{t+s}(q),\widehat{\psi}^1_t(\q))\to 0\text{ as }t\to \infty $$

Let  $s_1 = d_{\OG_1}(q,\q)$.
Fix $\delta_1>0$, $\delta_1 << 1, \delta_1 < s_1$.
By the above limit we can find $s_2>0$ such that  $d_{\OG_1}(\widehat{\psi}^1_{t+s}(q),\widehat{\psi}^1_t(\q))\leq \delta_1$ whenever $t>s_2$.

For any $t'>s_2$, consider the four points $q,q_{t'}=\widehat{\psi}^1_{t'+s}(q),\q$ and $\q_{t'}=\widehat{\psi}^1_{t'}(\q)$. By the triangle inequality,
\begin{equation}\label{ieq}
    d_{\OG_1}(q,q_{t'})\leq d_{\OG_1}(q,\q)+d_{\OG_1}(\q,\q_{t'})+d_{\OG_1}(\q_{t'},q_{t'})
\end{equation}

Note that $d_{\OG_1}(q,q_{t'})=\ell_{\OG_1}(q_{t'})$ as $q$ and $q_{t'}$ lie on the same flow ray and by Lemma \ref{qgo}, flow segments are length minimizing in $\widehat{M}$ w.r.t. the path metric of $\OG_1$, similarly $d(\q,\q_{t'})=\ell_{\OG_1}(q_{t'})$. Moreover by our assumption $d_{\OG_1}(q,\q) = s_1$, \  $d_{\OG_1}(q_{t'},\q_{t'})<\delta_1$. Replacing in the above inequality \ref{ieq}, we get

\begin{equation}\label{inq2}
\begin{split}
          \ell_{\OG_1}(q_{t'}) & \leq  d_{\OG_1}(q,\q)+d_{\OG_1}(\q,\q_{t'})+d_{\OG_1}(\q_{t'},q_{t'}) \\
                     & \leq  s_1+\ell_{\OG_1}(\q_{t'})+\delta_1\\
                     & =  \ell_{\OG_1}(\q_{t'})+2s_1\text{ for any }t'>s_2
\end{split}
\end{equation}

As $\q\in\OT\cap\{yt\}-$plane, by Lemma \ref{unstable} we know there
are global $K, k > 0$, so that $\ell_{\OG_1}(\q_{t'})\leq K\D_{\OG_1}(\q_{t'},\OT)+k$. Applying it in \ref{inq2},
\begin{equation}\label{inq3}
    \ell_{\OG_1}(q_{t'})\leq \ell_{\OG_1}(\q_{t'})+2s_1 \leq  K\D_{\OG_1}(\q_{t'},\OT)+k+2s_1 \ \ \text{ when }t'>s_2
\end{equation}

Finally, suppose $a\in\OT$ is a point in $\OT$ that 
is  closest to $q_{t'}$, i,e $d_{\OG_1}(q_{t'},a)=\D_{\OG_1}(q_{t'},\OT)$. By the triangle inequality, we get, $d_{\OG_1}(\q_{t'},a)\leq d_{\OG_1}(q_{t},a)+d_{\OG_1}(q_{t'},\q_{t'})$. Moreover $\D_{\OG_1}(\q_{t'},\OT)\leq d_{\OG_1}(\q_{t'},a)$ as $a\in\OT$. Combining all these facts we conclude when $t'>s_2$,
\begin{equation}\label{inq4}
\begin{split}
    \D_{\OG_1}(\q_{t'},\OT) & \leq d_{\OG_1}(\q_{t'},a) \leq d_{\OG_1}(q_{t'},a)+d_{\OG_1}(\q_{t'},q_{t'}) \\
                   &= \D_{\OG_1}(q_{t'},\widehat{T}_0)+d(\q_{t'},q_{t'})\text{ as }d_{\OG_1}(q_{t'},a)=\D_{\OG_1}(q_{t'},\OT)\\
                   &= \D_{\OG_1}(q_{t'},\OT)+s_1\text{ as }d_{\OG_1}(\q_{t'},q_{t'}) = s_1 \text{ by assumption}
 \end{split}   
\end{equation}

Combining \ref{inq3} and \ref{inq4} we get,
\begin{equation}\label{inq5}
    \ell_{\OG_1}(q_{t'})\leq K\D_{\OG_1}(\q_{t'},\OT)+k+2s_1\leq K\D_{\OG_1}(q_{t'})+Ks_1+k+2s_1\text{ when }t'>s_2
\end{equation}

The above inequality proves Proposition \ref{key} for $t'>s_2$. If $t'\leq s_2$, then $\ell(q_{t'})=t'<s_2$. Adding the case when $t'\leq s_2$ in Inequality \ref{inq5} we conclude,
\begin{equation}\label{inq5.1}
    \ell_{\OG_1}(q_{t'})\leq K\D_{\OG_1}(q_{t'})+Ks_1+k+2s_1+s_2\text{ for all }t'\in[0,\infty)
\end{equation}

By renaming $C=K$ and $c=Ks_1+k+2S_1+s_2$, we rewrite the above Inequality \ref{inq5.1} as 
$$\boxed{\ell_{\OG_1}(q_{t'})\leq C\D_{\OG_1}(q_{t'})+c\text{ for all }t'\in[0,\infty)}$$
It completes the proof of Proposition \ref{key} for the point $q$.

\vskip .08in
We still need to argue why we can find constants $C>1$ and $c>0$ which work for all $q\in \widehat{T}_0\setminus N_{\delta}(\widehat{\C}_1\cup\widehat{\C}_2)$.  We need to find $s_1$ and $s_2$ big enough, such that Inequality \ref{inq5.1} holds for all flow rays intersecting $\widehat{T}_0\setminus N_{\delta}(\widehat{\C}_1\cup \widehat{\C}_2)$.

Consider a fundamental domain in $\OOT_0$ which quotient downs on 
the torus in $M_1$.  w.l.o.g. we assume that the fundamental domain is bounded by the planes $t=0$ and $t=1$ and we call it $\overline{T}_{0,1}$. Next we consider the compact set $\mathcal{S}$ which is the closure of the set $\widehat{T}_{0,1}\setminus N_{\delta}(\widehat{\C}_1\cup \widehat{\C}_2)$. 

As $\mathcal{S}$ is compact, it has finite radius w.r.t the metric $d_{\OG_1}$.
Recall the definition of $s_1$. It is $s_1 = d_{\OG_1}(q, \q)$.
Here $\q = (D_1 \cup D_2)  \cap L_q$, where $L_q$ is a leaf of $\widehat{\F}$
and $D_1 \cup D_2$ are the intersections of the $yt$-plane with 
$\OOT_0$.
Since $q$ is in a compact set $\mathcal{S}$ it follows that
$\q$ is also in a compact set. It follows that $s_1$ is globally
bounded.

Now we consider $s_2$. Given $q$, the value $s$ was defined
so that $\widehat{\psi}^1_s(q)$ and $\q$ lie in the same
strong stable leaf of $\widehat{\psi}^1_t$. Again since
$\mathcal{S}$ is compact, and $\q$ is in a compact set,
it follows that the values of $s$ as a function of $q$ are
also globally bounded in $\mathcal{S}$. Then there is a
global $s_2 > 0$ so that 
$d_{\OG_1}(\widehat{\psi}^1_{t+s}(q), \widehat{\psi}^1_t(\q)) 
< \delta_1$ for all $t > s_2$.

This shows that $s_1, s_2$ can be chosen globally bounded
for $q$ in the fundamental domain $\mathcal{S}$.
Since it is a fundamental domain, this shows that
$s_1, s_2$ can be chosen globally bounded.

This finishes the proof of Proposition \ref{key}.
\end{proof}

\begin{remark}
The reason behind considering the $\delta$-neighbourhood of $\C_1\cup\C_2$ is to use the compactness of the set $\mathcal{S}$, the compactness is used to determine the universality of the constants $s_1$ and $s_2$. $\widehat{T}_{0,1}\setminus{\widehat{\C}_1\cup\widehat{\C}_2}$ is not compact. 

\end{remark}

Now we extend Proposition \ref{key} in the universal cover of $\widetilde{M}_1$ w.r.t. the lifted Riemannian metric $\widetilde{G}_1$. Note that $\partial \TM_1$ is the lift of the torus $\partial M_1$, and it is a collection of infinitely many planes homeomorphic to $\RR^2$. We first re-define the notations as follows:

For a point $q\in\TM_1$, suppose the flow line through $q$ intersect a component of $\partial\TM$, say $\widetilde{T}_q$, then
\begin{itemize}
    \item $\D_{\TG_1}(q,\widetilde{T}_q)$  denotes the distance between $q$ and $\widetilde{T}_q$ w.r.t. the path metric induced by $\TG_1$ on $\TM_1$. 

    \item $\ell_{\TG_1}(q)$ is the length of the flow line segment connecting $q$ and $\widetilde{T}_q$ w.r.t. $\TG_1$.

\end{itemize}
Consider the lift of the neighbourhood $N_{\delta}(\C_1\cup\C_2)$ in $\TM_1$, we denote it as $N_{\delta}(\widetilde{\C}_1\cup\widetilde{\C}_2)$. We can restate the Lemma \ref{key} in $\TM_1$ as follows:

\begin{lemma}\label{key1}
There exists $C > 1$ and $c > 0$ satisfying the following:
Let $q\in \TM_1$ such that the flow ray through $q$  which intersects $\partial\TM_1$, say at the boundary component $\widetilde{T}_q$. If $\gamma_q$ intersects $\widetilde{T}_q$ on the region $\widetilde{T}_q\setminus N_{\delta}(\widetilde{\mathcal{C}}_1\cup\widetilde{\mathcal{C}}_2)$, 
then
$$\ell_{\TG_1}(q)\leq C\D_{\TG_1}(q, \widetilde{T}_q) +d$$
\end{lemma}

\begin{proof}
Note that $\OM_1$ is an intermediate cover of $M_1$, hence  $\TM_1$ is the universal cover of $\OM_1$. For any two points $b_1$ and $b_2$ in $\TM_1$ if $\bar{b}_1,\bar{b}_2 $ denotes the projection of $b_1, b_2$ in $\OM_1$, then
$$d_{\OG_1}(\bar{b}_1,\bar{b}_2)\leq d_{\TG_1}(b_1,b_2)$$
As $\TM_1$ is the universal cover of a compact manifold, 
for any point $q\in\TM_1$ there exists a point $q^*\in \widetilde{T}_q$ such that $\D_{\TG_1}(q,\widetilde{T}_0)=d_{\TG_1}(q,q^*)$.
In addition there is a path in $\TM_1$ from $q$ to $q^*$ which realizes
this distance. This implies that if $\bar{q}\in\OM_1$ is the projection of the point $q\in\TM_1$, then $\D_{\OG_1}(\bar{q},\widehat{T}_q)\leq \D_{\TG_1}(q,\widetilde{T}_q)$. Moreover, the $t$-directions are unchanged in $\OM_1$ and $\TM_1$, hence $\ell_{\OG_1}(\bar{q})=\ell_{\TG_1}(q)$.  Combining all the information we get,

$$\ell_{\TG_1}(q)=\ell_{\OG_1}(\bar{q})\leq C\D_{\OG_1}(\bar{q},\widehat{T}_{\bar{q}})+c\leq C\D_{\TG_1}(q,\widetilde{T}_q)+c$$
\end{proof}

Now we extend Lemma \ref{key1} to $\M_1= M_1\cup(\partial{M_1}\times [0,1])$. The manifold $\N$ is the union of the collection of $\{(\M_1, \Psi_t^1, \mathcal{G}_1);(\M_2, \Psi_t^2, \mathcal{G}_2);(\M_3, \Psi_t^3, \mathcal{G}_3);...;(\M_n, \Psi_t^n, \mathcal{G}_n)\}$ where the plugs intersect each other along their boundary components. As described in Section $3$, we extend the Riemannian metrics $G_i$ from $M_i$ to $\M_i$ and the extended metric is $\G_i$.
The induced Riemannian metric in $\N$ is $\G$. In particular

$$\G|_{M_i}=G_{i}|_{M_i}$$ 

In the rest of the article we consider the path metric $d_{\GG}$ induced from the Riemannian metric $\GG$ on the whole manifold $\NN$. For two points $p_1,p_2\in\NN$,
$$d_{\GG}(p_1,p_2)=\text{minimum}\{ \text{ length} \ \sigma|\ \sigma\text{ is a path connecting} p_1,p_2\in\NN\}$$

We extend the flow $\psi^{i}_t$ from $M_i$ to $\M_i=M_i\cup (M_i\times [0,1])$ as a product flow  (topologically) on $\partial M_1\times[0,1]$, and the extended flow is denoted by $\Psi^i_t$. It is clear that the foliation $\LE^{ws}(\A_i)\cap \partial M_i$ on $\partial M_i$ also extends to $\partial M_1 \times [0,1]$,
and hence to $\partial \M_i$. 
The neighbourhood $N_{\delta}(\widetilde{\mathcal{C}}_1\cup\widetilde{\mathcal{C}}_2)$ is also carried by the extended flow 
on the new boundary $\partial \MM_1$. 
In the universal cover $\NN$, we denote this new set by $N'_{\delta}(\widetilde{\mathcal{C}}_1\cup\widetilde{\mathcal{C}}_2)$.



    



 As before we define, 
\begin{definition}
Suppose $q\in {\MM}_1$ such that the flow line $\gamma_q$ through $q$ intersects $\partial\MM_1$ at the boundary component $\widetilde{\T}_q$.  Then let
\begin{itemize}

    \item $\D_{\GG}(q,\widetilde{\T}_q)$ denote the distance between $q$ and $\widetilde{\T}_q$ w.r.t. the path metric induced by $\GG_1$. 
    \item $L_{\GG}(q) = L_{\GG_1}(q)$ denote the length of the flow line segment connecting $q$ and $\widetilde{\T}_q$ w.r.t. $\GG$. 
   
\end{itemize}
\end{definition}
\begin{lemma}\label{key2}
Let $q\in \widetilde{\M}_1$ such that the flow line through $q$ intersects $\partial\MM_1$ at the boundary component $\widetilde{\T}_q$. If $\gamma_q$ intersects $\widetilde{\T}_q$ on the region $\widetilde{\T}_q\setminus N'_{\delta}(\widetilde{\C}_1\cup \widetilde{\C}_2)$, then there exists $C_1>1$ and $c_1$ such that,
$$L_{\GG}(q)\leq C_1\D_{\GG}(q, \widetilde{\T}_q) +c_1$$
Moreover the constants $C_1$ and $c_1$  do not depends on $q$ or $\widetilde{\T}_q$. 
\end{lemma}
\begin{proof}
Again we will prove the lemma only for a component of $\partial \MM_1$, which
we will denote by $\widetilde{\T}_0$. The same result holds for all other components of $\partial\widetilde{\M}_i$.

Since the flow is a product in $\partial M_1 \times [0,1]$ then up
to changing $c_1$ to a bigger constant, we can assume
that $q \in \widetilde{M}_1$.

If $\gamma$ intersects $\widetilde{\T}_0$ then it also intersects 
a component $\widetilde{T}_0$ of $\partial \widetilde{M}_1$, because the flow $\Psi^1_t$ is a product in $\partial M_1 \times [0,1]$. As the distances are measured as minimum lengths of paths for both $\TG_1$ and $\GG$ and $\GG|_{\TM_1}=\TG_1$, we conclude
$$\D_{\TG_1}(q,\widetilde{T}_0) <  \D_{\GG}(q,\widetilde{\T}_0)$$

\noindent
Notice we are assuming that $q$ is in $\widetilde M_1$.
By the compactness of $\partial M_1\times [0,1]$, we can find $\epsilon>0$ (same as in Lemma \ref{qgo1}) such that for all $\gamma$ which intersects $\widetilde{\T}_0$,
\  $\text{length}_{\GG}(\gamma\cap(\widetilde{\OT}\times[0,1]))\leq \epsilon$.

Fix a flow ray $\gamma$, suppose $\gamma$ intersects $\widetilde{T}_0$ at $q_1$ and  $\widetilde{\T}_0$ at $q_2$. Then for any $p\in\gamma$, as $\TG_1|_{\MM_1}=\GG|_{\MM_1}$,
$$L_{\GG}(q)=\ell_{\TG_1}(q)+\text{length}_{\GG}(\gamma_{[q_1,q_2]})$$

By Lemma \ref{key1}, we know $\ell_{\TG_1}(q)\leq K\D_{\GG_1}(q,\widetilde{T}_0)+k$ and we have $\D_{\TG_1}(q,\widetilde{\T}_0)<\D_{\GG}(q,\widetilde{T}_0)$. By the definition of $\epsilon$, 
$\text{length}_{\GG}(\gamma_{[q_1,q_2]})\leq \epsilon$. Hence we conclude
\begin{equation}
    \begin{split}
        L_{\GG}(q)&=\ell_{\TG_1}(q)+\text{length}_{\GG}(\gamma_{[q_1,q_2]})\\
        & \leq \ell_{\TG_1}(q)+\epsilon \text{ as length}_{\GG}(\gamma_{[q_1,q_2]})\leq \epsilon\\
        & \leq C\D_{\TG_1}(q,\widetilde{T}_0)+c+\epsilon\text{ by Lemma \ref{key1} } \\
        & \leq C\D_{\GG}(q,\widetilde{\T}_0)+c+\epsilon\text{ as }\D_{\TG_1}(q,\widetilde{T}_0)\leq \D_{\GG}(q,\widetilde{\T}_0))
    \end{split}
\end{equation}

As all the constants $C,c$ and $\epsilon$ are independent of the flow line, we conclude that, for all $q$ in $\gamma$ such that $\gamma$ intersects $\widetilde{\T}_q\setminus N'_{\delta}(\widetilde{\C}_1\cup\widetilde{ \C}_2)$,
$$L_{\GG}(q)\leq C_1\D_{\GG}(q,\widetilde{\T}_q)+c_1$$
where $C_1=C$ and $c_1=d+\epsilon$ as defined above. 
\end{proof}

\begin{remark}
Lemma \ref{key1} says that every flow line which intersects the boundary components of $\widetilde{M}_1$ outside the `\textit{bad region}' goes away from the boundary component at a uniformly efficient rate
as $t\to\infty$. 

All other hyperbolic plugs, irrespective of `attracting' or `repelling', have the same type of property that if a flow ray intersects a boundary component outside the 'bad region', it goes 'away' at a uniformly
efficient rate from the boundary component when $t\to\infty$ (in case of attracting plugs) or $t\to-\infty$ (in case of repelling plugs). Moreover by taking the maximum over all the constants, we can fix global additive and multiplicative constants which work for all of the hyperbolic plugs. 
\end{remark}

We conclude this section with two remarks on the separating tori at the boundaries of the hyperbolic plugs, these tori play an important role in this article. 

\begin{remark}\label{tori}
\begin{enumerate}
    \item As remarked before, the boundary tori are \textit{incompressible}, i,e they are two-sided and injectively included in the fundamental group $\pi_1(\N)$.
    
    \item If $\widetilde{\T}$ is a component of the lift of some $\partial \MM_i$, then  by Theorem 1.1 of \cite{KL98}, see also Section 3.1 of 
\cite{Ngu19}, the following happens:
 $\widetilde{\T}$ is \textit{quasi-isometrically embeded} in the universal cover $\NN$. We make it more precise as follows: Consider the lift of a separating torus, say $\widetilde{T}$. By restricting the Riemannian metric on $\widetilde{\T}$, we can consider the path metric on $\widetilde{\T}$ induced by the restriction, we call it $d_{\widetilde{\T}}$. Then there exists $k_0, k_1$ such that the inclusion map $i:(\widetilde{\T},d_{\widetilde{\T}})\rightarrow (\NN,d_{\GG})$ is a $(k_0,k_1)$-quasi-isometric embedding. 
We can choose $k_0, k_1$ so that it works for any such 
$\widetilde{\T}$.

\end{enumerate}
\end{remark}

\subsection{Quasigeodesic behavior in the whole manifold:}\label{S5.2}

Now we are ready to prove the main theorem, that is, the flow lines of $\widetilde{\Psi}_t$ are uniform quasigeodesics in $\NN$ w.r.t. $d_{\GG}$, the path metric induced by $\GG$. As before, we assume that $\N$ is made of the collection of hyperbolic plugs $\{\M_1,\M_2,...,\M_n\}$. Each $\MM_i$ is a manifold with boundary such that $\partial \MM_i$ is a collection of separating planes homeomorphic to $\RR^2$ and properly embeded into $\NN$.

We first prove that if a flow line or flow ray is fully contained in the universal lift of a single plug $\M_i$ then it is a quasigeodesic. 

\begin{lemma}\label{qgr}
There exits $a_3>1$ and $a_4>0$ such that if $\gamma$ is a flow ray or flow line fully contained in a single hyperbolic plug $\MM_i$ then it is a $(a_3,a_4)$-quasigeodesic w.r.t. the metric $d_{\widetilde{\G}}$.
\end{lemma}
\begin{proof} 
As before, w.l.o.g. we prove the result only on $\MM_1$.

In Lemma \ref{qgo1}, we have proved that the flow lines or flow rays are quasigeodesic w.r.t. $d_{\GG_1}$, the path metric of the restriction of $\GG$ on $\MM_1$. In this lemma we need to extend the result on the whole manifold $\NN$ w.r.t. $d_{\GG}$. 

Consider two points $p_1$ and $p_2$ on a flow ray or flow line $\gamma\subset\MM_1$. Suppose $\sigma:[0,1]\rightarrow \NN$ is a minimal path w.r.t. $\GG$ in $\NN$ connecting $p_1$ and $p_2$. 
In other words $length(\sigma) = d_{\GG}(p_1,p_2)$.
We argue that the length of $\sigma$ can be approximated by a curve fully contained in $\MM_1$, without
too much increase in length.
 Suppose $\sigma$ exits $\MM_1$ through a boundary component $\widetilde{\T}$ at $0<v_1<1$, i,e $\sigma(v_1)\in \widetilde{\T}$. As $\widetilde{\T}$ is a separating plane in $\NN$, $\sigma$ must re-enter $\MM_1$ at some 
first $v_2$ with $0<v_1<v_2<1$.   

By Remark \ref{tori}, $\widetilde{\T}$ is $(k_0,k_1)$-quasi-isometrically embeded in $\NN$, hence we can find a path $\bar{\sigma}:[v_1,v_2]\rightarrow\widetilde{\T}$ such that 
$$\frac{1}{k_0}\text{length}_{\GG}(\bar{\sigma}_{[v_1,v_2]})-k_1\leq \text{length}_{\GG}(\sigma|_{[v_1,v_2]})\leq k_0\ \text{length}_{\GG}(\bar{\sigma}_{[v_1,v_2]})+k_1$$

In the definition of $\sigma$, we can replace $\sigma|_{[v_1,v_2]}$ with $\bar{\sigma}|_{[v_1,v_2]}$. 
Now fix $a_6 > 0, a_6 << 1$ so that any segment in the
image of $\sigma$ with endpoints in $\partial \MM_1$ and 
interior outside $\MM_1$, and length $< a_6$ can be pushed
into $\partial \MM_1$ to a segment of length at most $2 a_6$.
On $[0,1]$, there can exist only finitely many closed intervals on which $\sigma$ goes out of $\MM_1$ and with length $> a_6$. 
Also replace all of these intervals with minimal curve on respective boundary components as described before, using the quasi-isometry
constants $k_0, k_1$. Choosing $k_0 > 2$, then 
 we get a curve $\sigma':[0,1]\rightarrow\MM_1\subset\NN$ with 
same endpoints as $\sigma$, and such that 
$$\frac{1}{k_0}\text{length}_{\GG}(\sigma')-k_1\leq \text{length}_{\GG}(\sigma)\leq k_0\ \text{length}_{\GG}(\sigma')+k_1$$

Note that $\GG_1=\GG|_{\MM_1}$. In particular, as $\sigma'$ is a path contained in $\MM_1$ and connecting $p_1$ and $p_2$, hence 
$$d_{\GG_1}(p_1,p_2)\leq\text{length}_{\GG_1}(\sigma')=\text{length}_{\GG}(\sigma')$$

Hence combining Lemma \ref{qgo1} and the above inequality,

\begin{equation*}
    \begin{split}
        \text{length}_{\GG_1}(\gamma_{[p_1,p_2]})&\leq d_{\GG_1}(p_1,p_2)+\epsilon\\
                             & \leq \text{length}_{\GG_1}(\sigma')+\epsilon\\
                             & = \text{length}_{\GG}(\sigma')+\epsilon\\
                             & \leq k_0 \text{length}_{\GG}(\sigma)+ k_0 k_1+\epsilon\\
                             & = k_0 d_{\GG}(p_1,p_2)+k_0k_1+\epsilon
    \end{split}
\end{equation*}

As $\GG_1=\GG|_{\MM_1}$, $\text{length}_{\GG_1}\gamma_{[p_1,p_2]}=\text{length}_{\GG_1}\gamma_{[p_1,p_2]}$; by renaming $a_3=k_0$ and $a_4=k_0 k_1+\epsilon$, we conclude the proof of the lemma.
\end{proof}

The following is the main result of this article:

\begin{theorem} \label{main}
There exists $C_0>1$ and $c_0>0$ such that each flow line of $\widetilde{\Psi}_t$ in $\NN$ is a $(C_0,c_0)$-quasigeodesic w.r.t. the metric $\GG$. 
\end{theorem}

\begin{proof}

There exists exactly two types of flow lines $\gamma$ of $\widetilde{\Phi}_t$ in $\NN$: 
\begin{enumerate}
\item either $\gamma$ is contained in one of $\MM_i$s. In this case $\gamma$ is contained in the attractor (or repeller) inside $\MM_i$. 
\item or $\gamma$ intersects the boundary of one of $\MM_i$s. As every boundary component is shared by exactly two plugs, $\gamma$ intersects two adjacent plugs, say $\MM_i$ and $\MM_j$. In this case $\gamma$ is subdivided in two rays and each of $\MM_i$ and $\MM_j$ contains exactly one subray, say $\gamma^+$ and $\gamma^-$.  
\end{enumerate}

By Lemma \ref{qgr} we know that each flow lines which are entirely contained in one of $\MM_i$s are uniform quasigeodesics, i,e all the first type of flow lines in the above list are $(a_3,a_4)$-quasigeodesics.

If $\gamma$ intersects a common boundary component of $\MM_i$ and $\MM_j$, Lemma \ref{qgr} says that both of the forward subray $\gamma^+$ and the backward subray $\gamma^-$ are quasigeodesics. But concatenation of two quasigeodesic is not necessarily a geodesic and that is the main obstacle in this proof. Next we show that, in our case concatenation of two quasigeodesic flow ray is a quasigeodesic and the key ingredient of the proof is Lemma \ref{key2}. 

\begin{lemma}\label{lf}
Every flow line which intersects a boundary 
component of a hyperbolic plug is a quasigeodesic. 
\end{lemma}

\begin{proof}
Up to reindexing the $\MM_i$, we can assume that 
there are lifts $\MM_1, \MM_2$ of $\M_1, \M_2$ respectively, so that
$\gamma$ intersects $\partial \MM_1$ at the boundary component $\widetilde{\T}_0$ and suppose $\widetilde{\T}_0$ is a common boundary component of $\MM_1$ and $\MM_2$. We can assume that $\MM_1$ is an attracting plug, so it contains the forward ray $\gamma^+$ and $\MM_2$ must be a repelling plug and contains the backward flow ray $\gamma^-$. 

\begin{claim}\label{c1}
At least one of $\gamma^+$ or $\gamma^-$ intersects $\widetilde{\T}_0$ outside the `bad region'.
More precisely, at least one of the following is true:
\begin{itemize}
    \item For all $q\in\gamma^+\subset\MM_1$
    $$L_{\GG}(q)\leq C_1\D_{\GG}(q,\widetilde{\T}_0)+c_1 $$

    \item Or, $\gamma^-$ satisfies the same property, i,e. 
for all $q\in\gamma^-\subset\MM_2$,
    $$L_{\GG}(q)\leq C_1\D_{\GG}(q,\widetilde{\T}_0)+c_1 $$
    
\end{itemize}
It is possible that both of the flow rays $\gamma^+$ and $\gamma^-$ satisfy 
the property.
\end{claim}
\begin{proof}
 We can view $\widetilde{\T}_0$ as a boundary component of $\MM_1$, and
we denote it by $\T_{\MM_1}$ to emphasize this.
In the same way we can view $\widetilde{\T}_0$ as
a boundary component of $\MM_2$, which we denote by $\T_{\MM_2}$.
These boundary components are attached to each other by a map
 $\widetilde{\Omega}:\widetilde{\T}_{\MM_1}\rightarrow\widetilde{\T}_{\MM_2}$. 

Consider the neighbourhood $N'_{\delta}(\widetilde{\C}_1\cup\widetilde{\C}_2)$ on $\widetilde{\T}_{\MM_1}$ as described in Lemma \ref{key2}, lets rename it $\mathfrak{N}_{\delta}$, this is the `bad region' on $\widetilde{\T}_{\MM_1}$. Similarly there is another bad region, say $\mathfrak{N}_{\delta'}$ on $\widetilde{\T}_{\MM_2}$.

Note that we can choose $\delta$ (and $\delta'$) in the Lemma \ref{key} small enough, so that $\widetilde{\Omega}(\mathfrak{N}_{\delta})\cap \mathfrak{N}_{\delta'}=\emptyset$. In other words, every flow line $\gamma$ which intersects $\widetilde{\T}_0=\widetilde{\T}_{\MM_1}\sqcup\widetilde{\T}_{\MM_2}/\sim$ intersects at least one of the regions $\widetilde{\T}_{\MM_1}\setminus \mathfrak{N}_{\delta}$ or $\widetilde{\T}_{\MM_2}\setminus \mathfrak{N}_{\delta '}$. If $\gamma$ intersects $\widetilde{\T}_{\MM_1}\setminus \mathfrak{N}_{\delta}$ then the subray $\gamma^+$ satisfies the claim by Lemma \ref{key2}, and similarly if $\gamma$ intersects $\widetilde{\T}_{\MM_2}\setminus \mathfrak{N}_{\delta'}$ the subray $\gamma^-$ satisfies a similar property. 
\end{proof}



\noindent
{\bf {Continuation of the proof of Lemma \ref{lf}.}}

As before, let $\gamma^+\subset \MM_1$ and $\gamma^-\subset \MM_2$.
We assume  w.l.o.g, that $\gamma^+\subset\MM_1$ satisfies Claim \ref{c1}.

Take two points $q,q'\in\gamma$. As each of the subrays of $\gamma^+$ and $\gamma^-$ are uniform quasigeodesics by Lemma \ref{qgr}, we can conclude that the flow segment joining $q$ and $q'$ is a quasigeodesic if either $q,q'\in\gamma^+$ or $\q,q'\in\gamma^-$. Hence we assume $q\in\gamma^+$ and $q'\in\gamma^-$. 

As $\widetilde{\T}_0$ is a separating plane on $\NN$, $\gamma$ must intersect $\widetilde{\T}_0$. We can conclude that that for fixed $q' \in \gamma^-$:
$$\D_{\OGG}(q,\OTT_0)<d_{\OGG}(q,q')\text{ for all }q\in\gamma^+$$

 Let $\gamma$ intersect $\widetilde{\T}_0$ at $q_1$. 
Now fix $q \in \gamma^+$. We break the flow segment $\gamma_{[q,q']}$ as $\gamma_{[q,q']}=\gamma_{[q,q_1]}*\gamma_{[q_1,q']}$.  There are two possible cases:

\vskip .1in
\noindent
\textbf{Case I:} Suppose $\text{length}_{\GG}(\gamma_{[q,q_1]})\geq\text{length}_{\GG}(\gamma_{[q_1,q']})$.

Note that $L_{\GG}(q)=\text{length}_{\GG}(\gamma_{[q,q_1]})$. By our assumption 
$$\text{length}_{\GG}(\gamma_{[q,q']})\leq2\text{length}_{\GG}(\gamma_{[q,q_1]})=2L_{\GG}(q).$$

As we have assumed that $\gamma^+$ satisfies Claim \ref{c1},
\begin{equation}\label{inq8}
    \text{length}_{\GG}(\gamma_{[q,q']})\leq2L_{\GG}(q)\leq 2C_1\D_{\GG}(q,\widetilde{\T}_0)+2c_1
\end{equation}
We have observed that $d_{\GG}(q,q')>\D_{\GG}(q,\widetilde{\T}_0)$ as any curve joining $q$ and $q'$ also intersects $\widetilde{\T}_0$. Hence replacing in the previous Inequality \ref{inq8}, we conclude if $q,q'\in\gamma$ and $\text{length}_{\GG}(\gamma_{[q,q_1]})\geq\text{length}_{\GG}(\gamma_{[q_1,q']})$, then
\begin{equation}\label{inq12}
    \boxed{\text{length}\gamma_{[q,q']}<2L_{\GG}(q)\leq 2C_1\D_{\GG}(q,\widetilde{\T}_0)+2c_1\leq 2C_1d_{\GG}(q,q')+2c_1}
\end{equation}
This finishes the proof of uniform quasigeodesic behavior in
this case.

\vskip .1in
\noindent
\textbf{Case II:} Now we assume $\text{length}_{\GG}(\gamma_{[q,q_1]})< \text{length}_{\GG}(\gamma_{[q_1,q']})$.
\begin{figure}
    \centering
    \includegraphics[width=0.25\linewidth]{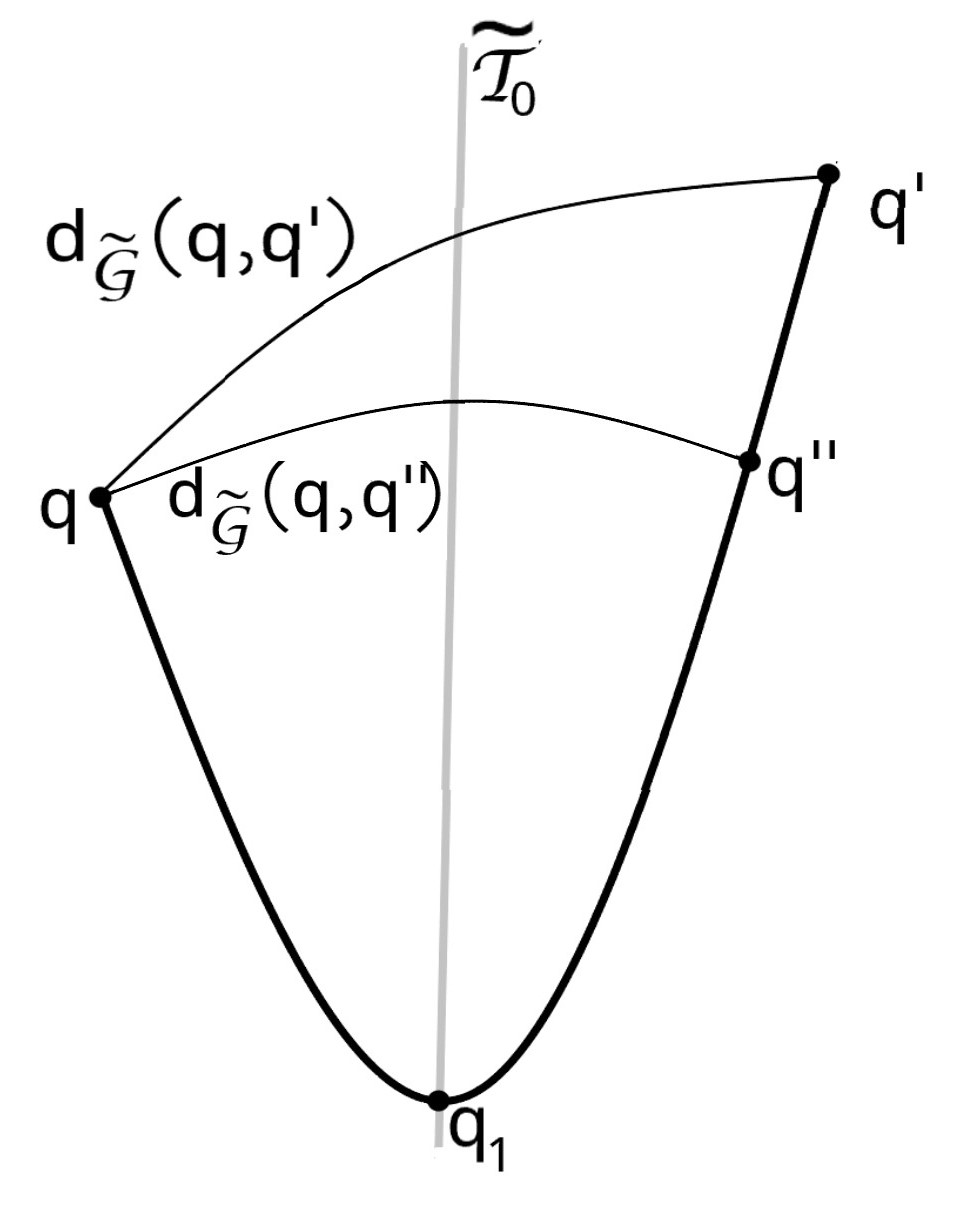}
    \caption{This depicts the situation in the universal cover.
	The curve with arrows in it is the flow line segment from
	$q$ to $q'$. The other two curves from $q$ to $q'$ and
	from $q$ to $q'$ are curves realizing the distance
	between these pairs of points.}
    \label{fig5}
\end{figure}
We first break $\gamma_{[q,q']}$ as $\gamma_{[q,q_1]}*\gamma_{[q_1,q'']}*\gamma_{[q'',q']}$ such that $\text{length}_{\GG}(\gamma_{[q,q_1]})= \text{length}_{\GG}(\gamma_{[q_1,q"]})$ as in Figure \ref{fig5}. 

Then by \ref{inq12} in case I, we conclude 
\begin{equation}\label{inq9}
    d_{\GG}(q,q'')\leq \text{length}_{\GG}(\gamma_{[q,q'']})\leq 2C_1\D_{\GG}(q,\widetilde{\T}_0)+2c_1
\end{equation}

Consider the points $q,q'$ and $q''$. The flow segment $\gamma_{[q'',q']}$ is an $(a_3,a_4)$  quasigeodesic segment  by Lemma \ref{qgr} as it is contained in a single plug $\MM_2$. Hence 
$$\text{length}_{\GG}(\gamma_{[q'',q']})\leq a_3 d_{\GG}(q',q'')+a_4$$
By the triangle inequality, we get $d_{\GG}(q',q'')\leq d_{\GG}(q',q)+d_{\GG}(q,q'')$. For the flow  segment $\gamma_{[q'',q']}$ we conclude
\begin{equation}\label{inq11}
\begin{split}
    \text{length}_{\GG}(\gamma_{[q'',q']})&\leq a_3d_{\GG}(q',q'')+a_4\\
    &\leq a_3[d_{\GG}(q',q)+d_{\GG}(q,q'')]+a_4\text{ \ \ by the
triangle inequality}\\
    &\leq a_3d_{\GG}(q,q')+a_3[2C_1\D_{\GG}(q,\widetilde{\T}_0)+2c_1]+a_4\text{ \ \ by \ref{inq9}}\\
    &=  a_3d_{\GG}(q,q')+2a_3C_1\D_{\GG}(q,\T)+2a_3c_1+a_4
\end{split}
\end{equation}
Finally adding \ref{inq9} and \ref{inq11} we conclude,
\begin{multline}
\text{length}(\gamma_{[q,q'']})+\text{length}(\gamma_{[q'',q']})\leq 
   2 C_1 \D_{\GG}(q,\widetilde{\T}_0)+2c_1+\\
    a_3d_{\GG}(q,q')+2a_3C_1\D_{\GG}(q,\widetilde{\T}_0)+2a_3c_1+a_4
\end{multline}

As every path connecting $q$ and $q'$ intersects $\widetilde{\T}_0$, we get $\D_{\GG}(q,\widetilde{\T}_0)\leq d_{\GG}(q,q')$. Replacing in the previous equation we conclude
\begin{equation}\label{inq13}
\boxed{
    \begin{split}
        \text{length}_{\GG}(\gamma_{[q,q']})&\leq 2C_1d_{\GG}(q,q')+
2a_3C_1d_{\GG}(q,q')+a_3 d_{\GG}(q,q')+2c_1+2a_3c_1 +a_4\\
        &=(2C_1+2a_3C_1+a_3)d_{\GG}(q,q')+2c_1+2a_3c_1+a_4
    \end{split}
    }
\end{equation}

We rename $C_0=2C_1+2a_3C_1+a_3$ and $c_0=2c_1+2a_3c_1+a_4$ and Equations \ref{inq12} and \ref{inq13} together imply that for any two points $q$ and $q'$ on $\gamma$, 
$$\text{length}_{\GG}{\gamma_{[q,q']}}\leq C_0d_{\GG}(q,q')+c_0$$
As $\gamma$ was chosen arbitrarily, this completes the proof that any flow line which intersects the common boundary component $\widetilde{\T}_0$ between $\MM_1$ and $\MM_2$ is a quasigeodesic. Moreover the same multiplicative and additive constants work for all flow line intersecting $\widetilde{\T}_0$.

In the beginning of the proof we fixed $\MM_1$ and $\MM_2$ and their common boundary component $\widetilde{\T}_0$, but the same method works for all other plugs which intersects along boundaries. The quasigeodesic constants differ for different choices of boundary components. Finally, as there are only finitely many plugs in $\N$, we can take maximums over all possibilities of boundary components in $\NN$ and we can choose global quasigeodesic constants for the flow lines which intersect any of the boundary components. 

This ends the proof of Lemma 
\ref{lf}.
\end{proof}
Lemma \ref{qgr} and Lemma \ref{lf} together imply that any flow line of $\widetilde{\Psi}_t$ in $\NN$ are uniformly quasigeodesic.

This completes the proof of Theorem \ref{main}.
\end{proof}

\begin{remark}
In the proof of the final theorem in Section $5.2$, we used two key ideas;
\begin{enumerate}
\item every flow ray of flow line contained in a single plug is quasigeodesic;
\item every flow ray which intersects a boundary component outside a narrow region eventually goes  away uniformly efficiently
from the boundary component it intersects.
\end{enumerate}
This suggests that the same general techniques used in this proof can be applied to study quasigeodesic behavior of flows in different contexts where these two properties hold. 
\end{remark}

\section{Comments on the Franks-Williams Manifolds $\N$ and a question}
As before we consider the decomposition of $\N$ as the union
$$\N=\M_1\cup \M_2 \cup ... \cup \M_n$$
where the components $\M_i$s intersects each other along their boundaries and the boundary components of $\M_i$ are homeomorphic to tori. We denote the collection of boundary components as $\{\T_j|j\in J\}$ where each $\T_j$ is a common boundary of two plugs from the collection $\M_1, \M_2,...,\M_n$. Now we again emphasize some properties of the manifold $\N$ and the collection of torus $\{\T_j|j\in J\}$:
\vskip 0.2in
\begin{enumerate}
    \item \textbf{$\N$ is irreducible:} It is easy to obtain from the construction that the pieces $\M_i$ are \textit{irreducible}, i,e every $2$-sphere bounds a $2$-ball in $\M_i$, and each 
    boundary component is incompressible. 
    We are attaching the pieces along boundaries homeomorphic to $2$-torus to get $\N$. Each such torus is incompressible in 
the particular piece containing it.
Hence the manifold $\N$ is irreducible.

    
    \vskip 0.1in
    \item \textbf{JSJ-decomposition of $\N$:} Every irreducible orientable three-manifold supports a JSJ-decomposition, that is, a collection of tori separating the manifold into atoroidal and/or Seifert fibered  pieces. In our case we can show that the collection $\{\T_j|j\in J\}$ is a minimal JSJ-decomposition on $\N$. The orientability condition can be satisfied by moving to a finite cover of $\N$ (if required). Note from the construction of $\N$ that each of the tori $\T_j$ is separating in $\N$. As explained in the previous item, each torus $\T_j$ is 
     \textit{incompressible}, i,e closed, 2-sided and canonically injects in the fundamental group of $\N$. To see that the collection $\{\T_j|j\in J\}$ is minimal, we claim that each piece in the decomposition w.r.t. $\{\T_j|j\in J\}$ is atoroidal, see detail in the item $3$ below.
     \vskip 0.1in
    \item \textbf{$\N$ is non hyperbolic and non Seifert fibered:} As there exits incompressible tori $\T_j$ in $\N$, if follows
that $\N$ cannot be a hyperbolic manifold. To see that $\M$ is also non Seifert fibered we consider the plugs $\M_1,\M_2,...,\M_n$ which are also the components of the torus decomposition, we note that each of the plugs are constructed by removing solid torus neighbourhoods from mapping tori of hyperbolic maps on $\TT$, i,e $\TT\times [0,1]/((x,1)\sim(\Phi(x),0))$ for some  $\Phi:\TT\to\TT$ homotopic to a hyperbolic map on $\TT$. 
    They are mapping tori of pseudo-Anosov homeomorphims of a torus minus finitely
    many disks. And it is well known that such mapping tori  are atoroidal, and hence non Seifert fibered. This now
    implies that these are the pieces of the JSJ decomposition of $\N$, and all pieces 
    are atoroidal.
    \vskip 0.1in
    
    \item \textbf{Existence of a \textit{non-positively curved} metric on $\N$:} As the manifold $\N$ is $Haken$, i,e
    it is  irreducible, and contains a closed incompressible surface, and it has one
    atoroidal piece, we can define a \textit{non-positively curved} Riemannian metric on $\M$ by \cite[Theorem 3.3]{Lee95}. This opens a potential new direction to explore as described below.
    
\end{enumerate}

\subsection{Future questions} By item (4) above, we can define a non-positively curved Riemannian metric on $\N$. Then the universal cover $\NN$ with the path metric induced by the non-positively curved Riemannian metric is $CAT(0)$. For a $CAT(0)$-space we can define different types of `boundaries at infinity',for example the \textit{Tits boundary} or the \textit{Morse boundary} and we can define topologies on the boundary at infinity, $\partial_{\infty}\NN$. For example, one question to ask is whether the weak-stable or weak-unstable foliations of $\Psi_t$ satisfy the \textit{Continuous Extension Property}:

\begin{question}
Suppose $F\in\FF^{ws}$ or $\FF^{wu}$ of $\widetilde{\Psi}_t$. It is known that leaves of the weak-stable or weak-stable foliations of Anosov flows are \textit{Gromov hyperbolic} and we can define the Gromov boundary $S^1(F)$ (it is a circle). Then does the inclusion $i:F\rightarrow\NN$ extends continuously to a map $\hat{i}:F\cup S^1(F)\rightarrow \NN\cup\partial_{\infty}\NN$?
\end{question}

Similar type of questions have been extensively studied  for Anosov flows on hyperbolic $3$-manifolds. A big difference in the case of Franks-Williams manifolds is that the Morse lemma is not true in general, but in some particular cases there is possibility that the quasigeodesic flow rays are boundedly away from actual geodesic rays. Hence this question is relevant for quasigeodesic Anosov flow on non-positively curved manifolds.





\end{document}